\pdfoutput=1
\documentclass[preprint,svgnames,dvipsnames]{elsarticle}

\usepackage[OT1]{fontenc}
\usepackage{amsmath,amsfonts, amssymb,amsthm}
\usepackage[numbers]{natbib}
\usepackage[colorlinks,citecolor=blue,urlcolor=blue]{hyperref}
\usepackage[left=3cm,right=3cm,top=3cm,bottom=3cm]{geometry}
\usepackage{algorithm}
\usepackage{algorithmicx}
\usepackage{algpseudocode}
\usepackage{bm}
\usepackage{graphicx}
\usepackage{dsfont}
\usepackage{tabularx}
\usepackage{stmaryrd}
\usepackage[svgnames,dvipsnames]{xcolor}
\usepackage{listings}
\usepackage{amssymb}
\usepackage{booktabs}

\newtheorem{lme}{Lemma}
\newtheorem{thm}{Theorem}
\newtheorem{dfi}{Definition}
\newtheorem{example}{Example}

\newtheorem{propo}{Proposition}

\newtheorem{rmk}{Remark}

\newcommand{\ie}{\textnormal{i.e.,}~}
\newcommand{\eg}{\textnormal{e.g.,}~}

\newcommand{\vs}{\textnormal{vs.}~}

\newcommand{\wrt}{\textnormal{w.r.t.}~}

\newcommand{\iid}{\textnormal{i.i.d.}~}
\newcommand{\myeqref}[1]{Eq.~(\ref{#1})}
\newcommand{\myfigref}[1]{Figure~\ref{#1}}

\newcommand{\mytabref}[1]{Table~\ref{#1}}

\newcommand{\E}[1]{\mathbb{E}\left[ #1 \right]}
\newcommand{\V}[1]{\mathbb{V}\left( #1 \right)}

\newcommand{\Sclos}{S^{\textnormal{clos}}}
\newcommand{\St}{S^T}
\newcommand{\Pmatrix}[1]{\begin{pmatrix} #1 \end{pmatrix}}
\newcommand{\PVext}{\textnormal{PV}^{\,\textnormal{0}}}
\newcommand{\PME}{\textnormal{PME}}
\newcommand{\PV}{\textnormal{PV}}
\newcommand{\Sh}{\textnormal{Sh}}

\title{Proportional marginal effects for global sensitivity analysis}

\author[a]{Margot Herin}
\author[b,c,d]{Marouane Il Idrissi}
\author[b,c]{Vincent Chabridon}
\author[b,c,d,e]{Bertrand Iooss}

\address[a]{Sorbonne Universit\'e, Laboratoire d'Informatique de Paris 6, 4 place Jussieu, 75005 Paris, France.}
\address[b]{EDF Lab Chatou, 6 Quai Watier, 78401 Chatou, France}
\address[c]{SINCLAIR AI Lab., Saclay, France}
\address[d]{Institut de Mathématiques de Toulouse, 31062 Toulouse, France}
\address[e]{Corresponding Author - Email: bertrand.iooss@edf.fr}

\begin{document}

\begin{frontmatter}

\begin{abstract}
Performing (variance-based) global sensitivity analysis (GSA) with dependent inputs has recently benefited from cooperative game theory concepts.
By using this theory, despite the potential correlation between the inputs, meaningful sensitivity indices can be defined via allocation shares of the model output's variance to each input. 
The ``Shapley effects'', i.e., the Shapley values transposed to variance-based GSA problems, allowed for this suitable solution. However, these indices exhibit a particular behavior that can be undesirable: an exogenous input (i.e., which is not explicitly included in the structural equations of the model) can be associated with a strictly positive index when it is correlated to endogenous inputs. In the present work, the use of a different allocation, called the ``proportional values'' is investigated. A first contribution is to propose an extension of this allocation, suitable for variance-based GSA. Novel GSA indices are then proposed, called the ``proportional marginal effects'' (PME). The notion of exogeneity is formally defined in the context of variance-based GSA, and it is shown that the PME allow the distinction of exogenous variables, even when they are correlated to endogenous inputs. Moreover, their behavior is compared to the Shapley effects on analytical toy-cases and more realistic use-cases.
\end{abstract}


\end{frontmatter}

\section{Introduction}\label{sec:intro}

When using phenomenological numerical models in science and engineering, the uncertainty quantification (UQ) process allows to consider and better quantify the various sources of uncertainties, most often by the way of probabilistic modeling \cite{ghahig17}. Global sensitivity analysis (GSA) is a key step of this process, aiming to understand the effects of each uncertain model input (or set of inputs) on the quantity of interest related to one (or more) output variable of interest obtained from the numerical model \cite{salrat08,ioosal17}. From a practical viewpoint, GSA aims at investigating four major settings \cite{davgam21}: (i.) model exploration, \ie investigating the input-output relationship; (ii.) factor fixing, \ie identifying non-influential inputs; (iii.) factor prioritization, \ie quantifying the most important inputs using quantitative importance measures; (iv.) robustness analysis, \ie quantifying the sensitivity of the quantity of interest with respect to probabilistic model uncertainty of the input distributions. In the present paper, one will more focus on the first three settings, without discussing much the fourth one.

Among a large panel of GSA indices, the variance-based sensitivity measures, also called ``Sobol' indices'' \cite{Sobol1990}, are derived from the functional analysis of variance (FANOVA) decomposition \cite{efrste81} between all the independent inputs. Thus, these indices enable to provide interpretable answers to some of the previously mentioned GSA settings. Let $Y=G(X)$ denotes the input-output relationship under study, with $G(\cdot) : \mathbb{R}^d \longrightarrow \mathbb{R}$ a deterministic (often black-box) numerical model, $Y$ a scalar output and $X = (X_1,\dots,X_d)$ a vector of $d$ scalar inputs. Moreover, let $\mathcal{P}(D)$ the set of all subsets of $D=\{1,\ldots,d\}$. For every subset of input $X_A=(X_i)_{i \in A}$, $A \in \mathcal{P}(D)$, the Sobol' indices are defined as follows:
\begin{equation}
S_A = \frac{\sum_{B \subset A} (-1)^{|A|-|B|} \V{\E{G(X)|X_B}}}{\V{G(X)}}
\end{equation}
where $|\cdot|$ denotes the number of elements in a subset.
If the inputs are assumed to be independent, thanks to the FANOVA decomposition, Sobol' indices lead to a well-defined allocation of an output's variance share (\ie $S_A$) to every subset of inputs $A \in \mathcal{P}(D)$. In this case, the variance's shares $(S_A)_{A \in \mathcal{P}(D)}$ sum up to one while being nonnegative. As the indices can be interpreted as proportions of the output variance, they allow to determine which inputs of a numerical model contribute the most to the variability of the output, or, on the contrary, to identify the ones that are not influential, and possibly which inputs interact with each other. Therefore, Sobol' indices can be directly used to answer to the factor fixing and factor prioritization settings (ii. and iii.).

However, in many applications, some inputs may have a statistical dependence structure, either initially imposed in their probabilistic modeling \cite{kurcoo06} or induced by physical constraints upon the input or the output space \cite{kuckly17,marcha21}. 
In these cases, estimating and interpreting Sobol' indices is not trivial as shown by many different analyses and interpretations proposed in the past (see \cite{ioopri19} or \cite{davgam21} for an overview of this topic). In order to circumvent this issue, \cite{owe14} proposed a new approach based on the ``Shapley value'' \cite{Shapley1951}, a solution concept developed in cooperative game theory and powerfully used in economic modeling. It consists in  distributing both gains and costs to several players working in coalition in an egalitarian way, ensuring that each player gains as much (or more) as they would have from playing individually. Therefore, based on Shapley values and Sobol' indices, \cite{owe14} proposed the so-called ``Shapley effects'' as new GSA indices in the context of dependent inputs. The underlying idea is to compute, similarly to a game involving coalition of players, the value assigned to a coalition of inputs $X_A$ as the explanatory power of a part of output variance. This value corresponds to the so-called ``closed Sobol' indices'' defined as:
\begin{equation}
\Sclos_A = \frac{\V{\E{G(X)|X_A}}}{\V{G(X)}}.
\label{eq:sclos}
\end{equation}
In the GSA context, the two main properties and advantages of the Shapley effects are the following: firstly, they cannot be negative; secondly, their sum is equal to one, even in the dependent inputs' case since they allow to bypass the intricate issue of variance decomposition \cite{owepri17,ioopri19}.
Let us remark that these two properties correspond to the two main desirability criteria for importance measures of linear regression models as reviewed in \cite{gro07}. Moreover, the egalitarian principle driving the allocation rule states that, in the independent inputs' case, an interaction effect is equally apportioned to each input involved in the interaction. Finally, several works have studied the Shapley effects estimation issues. Such estimates can be obtained via several techniques such as Monte Carlo-based algorithms \cite{sonnel16}, $k$-nearest neighbors \cite{brobac20} or M\"obius inverses \cite{plirab21}.

In \cite{ioopri19}, the Shapley effects have been claimed to be used for the factor fixing setting since an effect close to zero means that the input has no significant contribution to the variance of the output (neither by its interactions nor by its possible dependencies with other inputs).
However, another phenomenon, observed by \cite{ioopri19} and known as the ``Shapley's joke'' \cite{herili22}, proves that the factor fixing setting cannot be fully achieved with Shapley effects: an exogenous variable (i.e., which is not explicitly included in the structural equations of the model) can be granted a non-negligible share of the output variance, as soon as it is sufficiently correlated with endogenous inputs.
This means that Shapley effects do not respect the so-called ``exclusion property'' defined for the importance measures of linear regression models \cite{johleb04,gro07}.
This exclusion property states that, if an input's linear regression coefficient equals zero, then its importance measure should be zero too.

In the context of statistical learning, if $G$ is a linear regression model, an analogy can be made between the Sobol' indices and the squared value of the standardized regression coefficients (denoted by SRC$^2$). Moreover, the Shapley effects correspond to the so-called ``LMG measure'' (named after the authors' names, Lindeman-Merenda-Gold, see \cite{linmer80,bud93}), which partitions the explained variance percentage $R^2$ in the same way that is done by the Shapley-based allocation rule.
A weighted analog of LMG, called \emph{proportional marginal variance decomposition} (PMVD), has been proposed by \cite{Feldman2005} in order to respect the exclusion property. It is based on the proportional value allocation rule coming from cooperative game theory.
Its usefulness in relation to LMG has been described in details in \cite{gro07,gro15} and illustrated more recently in \cite{ilidrissi_iooss_chabridon_JdS_2021,ioocha22}.
In addition to the exclusion property, a more discriminating power between the influential inputs than the one obtained with the Shapley effects is also shown. Therefore, the PMVD is a good tool (in the linear regression context) to address the factor fixing setting.

In this paper, inspired on the one hand, by the work achieved in the linear regression context leading to the PMVD, and on the other hand, by the Shapley effects, we build and propose a set of novel sensitivity indices respecting the exclusion property and not restricted to the linear model case. To do so, the \emph{proportional marginal effects} (PME) are introduced by using a new variance decomposition, based on the proportional values concept \cite{Ortmann2000,Feldman2005}, which encompasses the ability to detect exogenous variables. For the sake of clarity, Table~\ref{table:analogy_methods} provides a first preliminary analogy to emphasize which category of problem one tries to address in the present paper.


\begin{table}[ht!]\label{table:analogy_methods}
\centering
\begin{tabular}{ll}\toprule
$R^2$ decomposition (linear regression) & $\V{Y}$ decomposition (GSA) \\
\midrule
SRC$^{2}$   & Sobol' indices\\
LMG & Shapley effects \\
PMVD  & \textbf{PME (proposed indices)} \\
\bottomrule
\end{tabular}
\caption{Analogy between linear regression importance measures ($R^2$ decomposition) and variance-based GSA.}
\end{table}

The rest of this paper is organized as follows. Section \ref{sec:coopGames_GSA} focuses on the interaction between GSA and cooperative game theory and the existing literature. The Shapley effects are recalled, as well as their main shortcoming: the inability to detect exogenous inputs. To that end, the notion of $L^2$\emph{-exogeneity} is formally defined. Then, Section \ref{sec:2} defines the proportional values and presents the main result of this paper, an extension allowing for well-defined novel GSA indices: the PME. It is additionally shown that these novel indices allow to detect exogenous inputs, while remaining inherently interpretable. Section \ref{sec:3} illustrates the behavior of the novel PME by using analytical formulas obtained for analytical forms of $G$. Section \ref{sec:4} briefly recalls several strategies for the estimation of PME and provides the results obtained on several more challenging numerical test-cases. Section \ref{sec:5} discusses several possible improvements as well as some perspectives about the proposed work. A few appendices provide extra materials such as information about reproducibility of numerical results (Appendix~\ref{app:software}) and proofs (Appendix~\ref{app:proofThm}).

Throughout this paper, let $\E{\cdot}$ and $\V{\cdot}$ denote the expectation and variance respectively. A \emph{coalition of players} is a subset of the \emph{grand coalition} denoted $D=\{1, \dots, d\}$. Moreover, $\forall A \subseteq D$, the restricted set of indices  $A\setminus \{ i \}$, for any $i \in A$, is denoted by $A_{-i}$. Additionally, for any $A \subseteq D$, $X_{D \setminus A}$ is denoted by $X_{\overline{A}}$. The distribution of the random inputs $X$ is generically denoted by $P_X$ and the marginal distribution of any subset of inputs $X_A$ for any $A \subseteq D$ is generically denoted by $P_{X_A}$. The spaces $L^2(P_{X_A})$, for any $A \subseteq D$, denote the spaces of measurable functions with finite second-order moments. When a function is referred to as being nonnegative (resp. positive), it entails that it takes values in $\mathbb{R}^+$ (resp. $\mathbb{R}^{+}_*$). Whenever reference is made to a model $G$, it is always implicitly assumed that $G \in L^2(P_X)$. In this paper, almost sure statements are followed by the acronym ``a.s.''. 


\section{Cooperative game theory for variance-based global sensitivity analysis}\label{sec:coopGames_GSA}

This section aims at reviewing the usefulness of cooperative game theory in the process of designing variance-based GSA indices. A particular class of allocations is presented: the random order model allocations, which contains the Shapley values. The Sobol' cooperative games are introduced, as a formalization of the analogy between players and inputs of deterministic models. The Shapley effects are presented as the application of Shapley values to a Sobol' cooperative game. The notion of dual of a cooperative game is also presented, and an analogy is drawn between backward-forward procedures and the random order model allocations. Finally, a specific Shapley effects' drawback (for factor fixing setting) is presented as a motivation for the proposed work: their inability to detect exogenous inputs.

\subsection{Analogy between allocation and variance-based GSA indices}\label{sec:analogyCoopGames_GSA}

A cooperative game is a tuple $(D,v)$ where $D = \{1,\dots,d\}$ is a set of $d$ players and  $v : \mathcal{P}(D) \rightarrow \mathbb{R}$ is the \emph{value function}, \ie an application that maps a value to every possible coalition of players. Usually, $v$ is assumed to be \emph{monotonically increasing}, meaning that, for any two sets $T$ and $A$ such that $T \subseteq A \in \mathcal{P}(D)$, one has $v(T) \leq v(A)$. In other words, the value of a coalition $A$ cannot be lower than the value of a sub-coalition $T \subseteq A$. In the following, cooperative games with monotonically increasing value functions are referred to as ``monotonic cooperative games''. Moreover, if the value function $v$ takes values in $\mathbb{R}^+_*$ (resp. in $\mathbb{R}^+$), the corresponding cooperative game is referred to as ``positive (resp. nonnegative) cooperative game``.

The analogy between the players $D$ of a cooperative game $(D,v)$ and the inputs $(X_i)_{i \in D}$ involved in a numerical model has been first used in \cite{owe14}. The author proposed to use, as a value function, the closed Sobol' indices recalled in \myeqref{eq:sclos}, allowing to define the Sobol' cooperative games.
\begin{dfi}[Sobol' cooperative game]\label{def:sobolGame}
Let $X=(X_1, \dots, X_d)^\top$ be random inputs, let $G \in L^2(P_X)$ be a model and denote $Y=G(X)$ the random output. A Sobol' cooperative game is the cooperative game with value function $\Sclos$ defined as follows:
\begin{align*}
    \Sclos\colon  \mathcal{P}(D) &\rightarrow \mathbb{R}^+ \\
     A &\mapsto \Sclos_A = \frac{\V{\E{Y \mid X_A}}}{\V{Y}}.
\end{align*}
The Sobol' cooperative game thus refers to the nonnegative, monotonic cooperative game $(D, \Sclos)$.
\end{dfi}
By analogy with the cooperative game theory paradigm, the choice of $\Sclos$ as a value function entails measuring the value of every subset of players $A \subseteq D$ as the variance of the best approximation of $Y$ on $L^2(P_{X_A})$, \ie $\V{\E{Y \mid X_A}}$.

One of the key aspects of cooperative games is the notion of \emph{allocation}. In general, allocations can be understood as a decomposition of the quantity $v(D)$ in $d$ elements, each one being allocated to a specific player. When it comes to Sobol' cooperative games, it translates to assigning a share of the output's variance $\V{Y}$ to each input in the model, with limited assumptions on the probabilistic structure between the inputs (in particular, no independence is assumed between the inputs). Formally, an allocation can be understood as a mapping $\phi$ that associates, to a cooperative game $(D, v)$, a real-valued vector $(\phi_1, \dots, \phi_d)^\top \in \mathbb{R}^d$.

The \emph{Shapley values}, are a particular example of allocations. For any cooperative game $(D, v)$, it is uniquely characterized as the allocation $\phi\bigl((D,v)\bigr)$ verifying a set of four distinct axioms:
\begin{enumerate}
    \item \textbf{Efficiency}: $\sum_{i=1}^d \phi_i = v(D)$; 
        
    \item \textbf{Symmetry}: $\forall i,j \in D$ with $i \neq j$, if $v(A\cup\{ i\}) = v(A\cup \{ j\})$ for all $A \in \mathcal{P}(D)$, then $\phi_i = \phi_j$;
        
    \item \textbf{Null player}:  $\forall i \in D$, if $v(A\cup\{ i\}) = v(A)$ for all $A \in \mathcal{P}(D)$, then $\phi_i=0$;
        
    \item \textbf{Additivity}: If two cooperative games $(D,v)$ and $(D,v')$ have Shapley values $\phi$ and $\phi'$ respectively, then the cooperative game $(D, v + v')$ has Shapley values $\phi_j + \phi_j'$ for $j \in D$.
\end{enumerate}
For any cooperative game $(D,v)$, its Shapley values can be expressed analytically, for any $i \in D$, as:
\begin{equation}
    \mbox{Shap}_i\bigl((D,v)\bigr) = \frac{1}{d}\sum_{A \subseteq D_{-i}} {d-1 \choose |A|}^{-1}\left[v(A \cup\{i\})-v(A)\right].
    \label{eq:ShapValues}
\end{equation}
This original formulation attributed to \cite{Shapley1951} can be interpreted as a weighted average, over every possible coalition $A$, of the contribution of a player $i$ to that coalition $A$. This contribution is quantified by the quantity $v(A \cup\{i\})-v(A)$, often called ``\emph{marginal contribution}'' of the player $i$ to the coalition $A$ in the literature. The weighting scheme can be understood as the proportion of permutations (or orderings) of $D$ such that $i$ appears after the players in $A$. While this interpretation can be hard to understand, defining the Shapley values in terms of players permutations allows for a better understanding of its underlying sharing mechanism, as it is done in the following.

A particular class of allocations, known as \emph{random order models} \cite{Weber1988, Feldman2007}, allows to define allocations based on orderings of players, instead of reasoning in terms of coalitions as in \myeqref{eq:ShapValues}. Let ${\cal S}_D$ be symmetric group on $D$ (the set of all permutations of $D$). Let $\pi = (\pi_1, \dots, \pi_d) \in {\cal S}_D$ be a particular permutation, and for any $i \in D$, denote $\pi(i)=\pi^{-1}_i$ its inverse (\ie the position of $i$ in $\pi$, such that $\pi_{\pi(i)} = i$). Then, one can define the following set of players, for any $i \in \{0, \dots, d\}$:
\begin{equation}
    C_i(\pi) = \{\pi_j : j \leq i  \}.
\end{equation}
$C_i(\pi)$ is the set of the $i$-th first players in the ordering $\pi$, with the convention that, for any permutation, $C_0(\pi) = \{\emptyset\}$. As an illustration, let $D=\{1,2,3\}$, and let $\pi = (2,1,3) \in {\cal S}_D$. Then, 
$$\pi(1) = 2, \quad \pi(2) = 1, \text{ and } \quad \pi(3) = 3.$$
Moreover, 
$$C_{\pi(1)}(\pi) = C_2(\pi) = \{1,2\}, \quad C_{\pi(2)}(\pi) = C_1(\pi) = \{2\}, \quad C_{\pi(3)}(\pi) = C_3(\pi) = \{1,2,3\}$$
As their names suggest, random order models endow ${\cal S}_D$ with a probabilistic structure. For any game $(D,v)$, the set of random order models allocations (or probabilistic allocations) contains every allocation $\phi \bigl( (D,v) \bigr)$ that can be written, for any $i\in D$, as:
\begin{align*}
    \phi_i &= \sum_{\pi \in {\cal S}_D} p(\pi)\left[ v\left(C_{\pi(i)}(\pi) \right) - v\left(C_{\pi(i)-1}(\pi) \right) \right]\\
    &= \mathbb{E}_{\pi \sim p}\left[ v\left(C_{\pi(i)}(\pi) \right) - v\left(C_{\pi(i)-1}(\pi) \right) \right]
\end{align*}
where $p$ is a probability mass function over the orderings of $D$. For a player $i$, its random order allocation can be interpreted as the expectation over the permutations $\pi$ of $D$ with respect to $p$, of the marginal contributions of $i$ to the coalitions formed by $C_{\pi(i)-1}(\pi)$. The random order model allocations are always \emph{efficient} and, when dealing with monotonic games, \emph{positive} (\ie $\phi_i \geq 0$ for any $i\in D$) \cite{Weber1988}. The Shapley values, in particular, can be expressed as a random order model allocation, under the particular choice of $p$ as a discrete uniform distribution over ${\cal S}_D$, which echoes \myeqref{eq:ShapValues}:
\begin{equation}
    \mbox{Shap}_i\bigl((D,v)\bigr) =\frac{1}{d!} \sum_{\pi \in {\cal S}_D} \left[ v\left(C_{\pi(i)}(\pi) \right) - v\left(C_{\pi(i)-1}(\pi) \right) \right].
    \label{eq:ShapValuesRO}
\end{equation}
Random order models allow to apprehend allocations \emph{dynamically} (see Section~\ref{sec:dualGame}), meaning that coalitions are formed regarding orderings, as opposed to the pure coalition point of view displayed in \myeqref{eq:ShapValues}. In this setting, Shapley values can then be understood as a maximum entropy a priori (\ie uniform over ${\cal S}_D$) about this dynamic. In the light of this equivalent expression, L. S. Shapley himself interpreted the Shapley values as ``\textit{[...] an a priori assessment of the situation, based on either ignorance or disregard of the social organization of the players}'' \cite{Shapley_1953}. 

When it comes to GSA, the Shapley values of the Sobol' cooperative game $(D,\Sclos)$ associated to a numerical model $Y=G(X_1,\dots,X_d)$ allow to define the so-called \emph{Shapley effects} \cite{owe14}. For any $i\in D$, they can be written as:
\begin{subequations}
	\begin{alignat}{3}
		\mbox{Sh}_i &:= \mbox{Shap}_i\bigl((D, \Sclos) \bigr)\\
		&=\frac{1}{d}\sum_{A \subseteq D_{-i}} {d-1 \choose |A|}^{-1}\left(\Sclos_{A \cup\{i\})}-\Sclos_A \right) \\
		&= \frac{1}{d!} \sum_{\pi \in {\cal S}_D} \left[ \Sclos_{C_{\pi(i)}(\pi)} - \Sclos_{C_{\pi(i)-1}(\pi)}\right].
	\end{alignat}
\end{subequations}
These indices, which have been extensively studied in \cite{sonnel16, owepri17, ioopri19}, are a great tool to quantify variable importance in the context of dependent inputs \cite{dav21}. They allow for a meaningful decomposition of $\V{Y}$ into positive shares attributed to each input, even in situations where the inputs are correlated. 

\subsection{Dual of a cooperative game}\label{sec:dualGame}
The notion of the \emph{dual} of a cooperative game is also of interest in the present paper. On the one hand, under the game theory paradigm presented previously, the aim of the value function $v$ is to quantify the ``value produced'' by a coalition of players (\eg the monetary value). On the other hand, the dual of a cooperative game focuses on the ``worth'', or ``bargaining power'' of a coalition, i.e., the shortfall in value due to a coalition \cite{Feldman2005, Feldman2007}. The dual of a cooperative game $(D,v)$ is usually denoted by $(D,w)$ where $w$ is defined, for any $A \in \mathcal{P}(D)$ as:
\begin{equation}
w(A) = v(D) - v(D \setminus A).
\end{equation}
The quantities $w(A)$ are often referred to as the \emph{marginal contribution of a coalition $A$ to the grand coalition $D$} in the literature, and is often interpreted as a measure of how crucial a coalition is in producing $v(D)$. For the sake of conciseness, in the following, one refers to $w(A)$ as the \emph{marginal contribution of the coalition $A$}. The dual $(D,w)$ of $(D,v)$ is also a cooperative game, and thus one can seek to construct relevant allocations for this game.

Following up this idea of dual game, one can draw a parallel between random order model allocations and the well-known ``forward'' and ``backward'' variable selection procedures. \myfigref{fig:sch_backward} illustrates this similarity. Formally, one can notice that, for a player $i$ and any permutation $\pi \in \mathcal{S}_d$, one has:
\begin{equation}
    w\left(C_{\pi(i)}(\pi)\right) - w\left(C_{\pi(i)-1}(\pi)\right) = v\left(D \setminus C_{\pi(i)-1}(\pi)\right) - v\left(D \setminus C_{\pi(i)}(\pi)\right).
\end{equation}
A random order model allocation of the dual of a cooperative game can be understood as the expected (with respect to a probability mass function $p$ over $\mathcal{S}_D$) marginal contribution of a player $i$ to the players that \emph{follows} in the orderings' dynamic, whereas for the initial cooperative game, it is the expected marginal contribution of $i$ to the players that \emph{precedes} in the orderings' dynamic.

\begin{figure}[!ht]
    \centering
    \includegraphics[width=\textwidth, trim={1.25cm 1cm 1.25cm 2.5cm}]{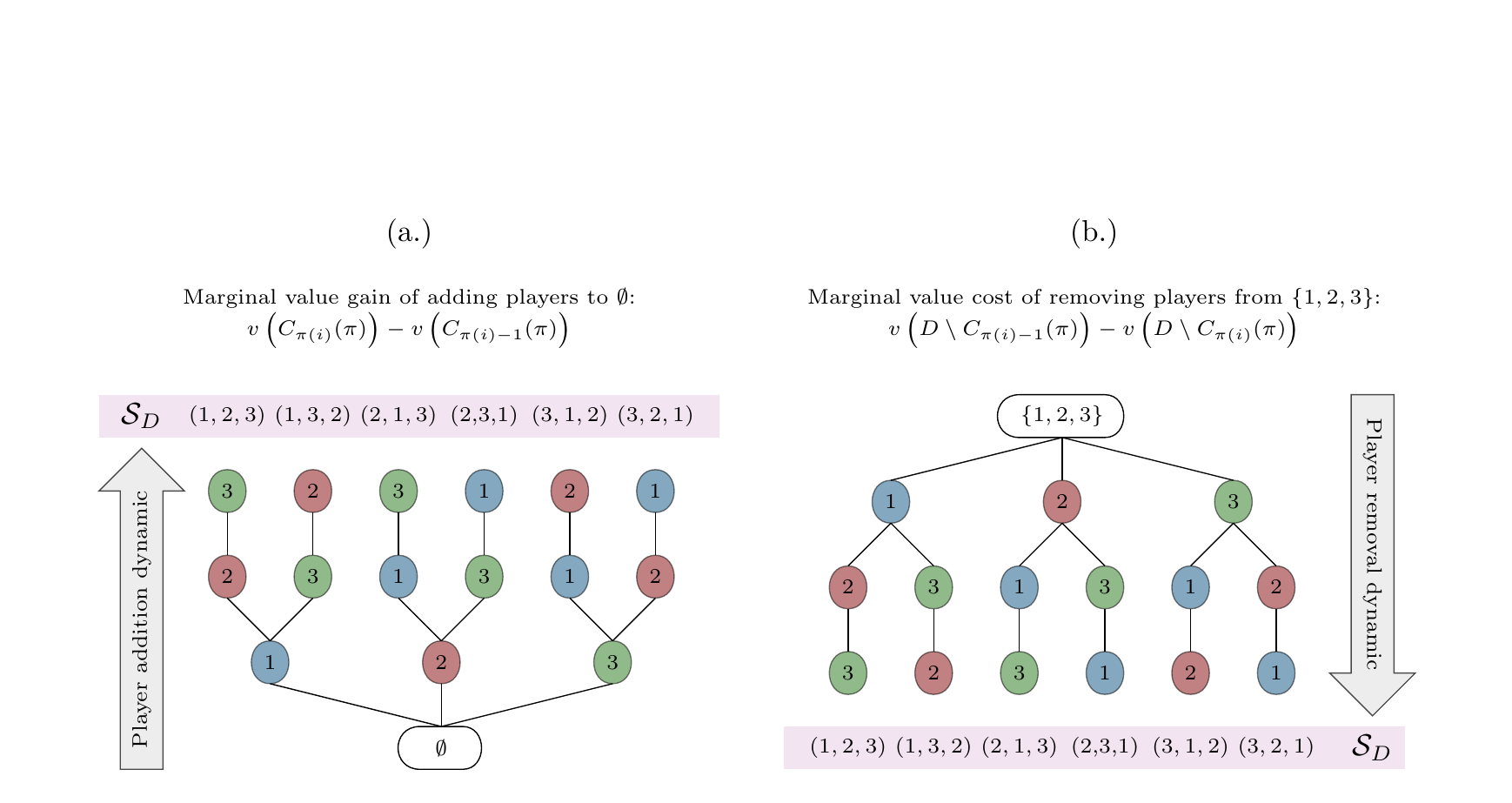}
    \caption{Analogy between random order model allocations and the forward-backward procedures for $D=\{1,2,3\}$: (a.) represents the allocation of a cooperative game as a forward procedure; (b.) illustrates the allocation of its dual as a backward procedure. The allocation of player $1$ (resp. player $2$ and $3$) is the expected marginal gain (for a cooperative game $(D,v)$) or cost (for its dual $(D,w)$) computed for the blue (resp. red and green) ordering positions, weighted according to a probabilistic distribution over $\mathcal{S}_D$.}
    \label{fig:sch_backward}
\end{figure}

The Shapley values of a cooperative game are equal to the ones of its dual (see, \cite{Funaki1996} Lemma 2.7). The dual of a Sobol' cooperative game $(D, \Sclos)$ is the nonnegative, monotonic cooperative game $(D, \St)$, where $\St$ denotes the total Sobol' indices, given for a model $Y = G(X)$, and for any $A \in \mathcal{P}(D)$, by
\begin{equation}\label{eq:Stot}
    \St_A = \frac{\E{\V{Y \mid X_{\overline{A}}}}}{\V{Y}}.
\end{equation}
The equivalence between the Shapley values of a Sobol' cooperative game and its dual has been highlighted by \cite{sonnel16}. Taking $\Sclos$ or $\St$ as a value function leads to the same Shapley effects, allowing for alternate estimation schemes. However, it is important to note that this is a particular property of the Shapley values, and it is not inherent to every random order model allocation.

\subsection{Detecting exogenous inputs}\label{sec:exoDetection}
As noted in \cite{herili22}, the main drawback (for factor fixing setting) of the Shapley effect is their behavior when dealing with exogenous (or spurious) inputs. Formally, exogenous inputs, in the context of variance-based GSA, can be defined as follows.
\begin{dfi}[$L^2$-exogeneity]\label{def:exogInput}
Let $X= (X_1, \dots, X_d)$ be random inputs of a model $G: \mathbb{R}^d \mapsto \mathbb{R}$ such that $Y=G(X)$, with $Y$ the random output. Additionally, it is assumed that any input cannot be expressed as a deterministic transformation of some (or all) of the others. Let $E \subset D$. The subset of random inputs $X_E$ are said to be ($L^2$-)exogenous to $G$ if, for any $B \subseteq E$, $\exists f \in L^2(P_{X_{\overline{B}}})$ such that:
$$Y = f(X_{\overline{B}}) \quad \text{a.s.}$$
\end{dfi}
In other words, inputs gathered in a subset $X_E$ are considered to be exogenous to a model $G$ if it is possible to characterize the random output $G(X)$ using a function depending on $X_{\overline{B}}$, for every $B \subseteq E$. This definition is rather intuitive: a set of inputs $X_E$ is considered as exogenous if one can find a function characterizing the random output which do neither involve the inputs in $E$, nor their possible interaction. It is important to notice that this definition is not too restrictive, since the exogenous inputs $X_E$ and $X_{\overline{E}}$ can still be correlated. In situation where the random inputs are correlated, the Shapley effects can allocate shares of variance to exogenous inputs. This phenomenon, called the \emph{Shapley's joke}, has been illustrated in \cite{ioopri19,herili22} through the following example.
\begin{example}[Shapley's Joke]\label{exa:ShapleyJoke}
Let $X = (X_1, X_2)^\top  \sim  \mathcal{N}\left(\begin{pmatrix}0\\0 \end{pmatrix}, \begin{pmatrix}1&\rho\\ \rho &1  \end{pmatrix} \right)$, $-1 < \rho < 1$,  and let the model be:
$$Y = G(X) = X_1.$$
The Shapley effects of the random inputs are given by
$$\mbox{Sh}_1 = 1- \frac{\rho^2}{2}, \quad \mbox{Sh}_2 = \frac{\rho^2}{2}.$$
\end{example}
Even if $X_2$ is an exogenous input, its Shapley effect is not zero as long as $\rho \not = 0$. While this behavior can be considered as valuable in a factor prioritization setting (effects due to correlation can be relevant), it can also be a drawback for spurious variable detection \cite{davgam21}. To overcome this drawback, other allocations can be considered. In particular, the \emph{proportional values} (PV) \cite{Ortmann2000} allow to detect exogenous variable while preserving the factor prioritization power and interpretability of cooperative game theory allocations.  

\section{From proportional values to proportional marginal effects}\label{sec:2}
The proportional values of positive cooperative games are a particular random order model allocation. They are introduced and extended to nonnegative games in order to be computed for Sobol' cooperative games. This extended allocation applied to the dual of Sobol' cooperative games are introduced as the \emph{proportional marginal effects} ($\PME$). It is then shown that the $\PME$ allow to detect exogenous inputs by granting them zero allocation.

\subsection{Proportional values as an alternative allocation strategy to Shapley values}\label{sec:PVbetterShap}
The PV of a cooperative game $(D,v)$ is a particular case of random order model allocation. PV arise from a particular case of chosen probability mass function \cite{Feldman2005}, and can also be characterized recursively \cite{Feldman1999, Ortmann2000}.
\begin{dfi}[Proportional values]\label{def:propVal}
Let $(D,v)$ be a positive, monotonic cooperative game, where $v:\mathcal{P}(D) \rightarrow \mathbb{R}^+_*$. The proportional values of $(D,v)$, denoted $\PV(D,v)\in \mathbb{R}^d$, are defined, for every $i \in D$, as a random order model allocation:
\begin{equation}
    \PV_i = \sum_{\pi \in {\cal S}_D} p(\pi)\left[ v\left(C_{\pi(i)}(\pi) \right) - v\left(C_{\pi(i)-1}(\pi) \right) \right]
    \label{eq:def_randord_PV}
\end{equation}
for the following particular choice of probability mass function over $\mathcal{S}_D$:
\begin{equation}
    p(\pi) = \frac{L(\pi)}{\sum_{\sigma \in {\cal S}_D} L(\sigma) }, \quad \text{where}\quad L(\pi) = \left(\prod_{j \in D} v\left( C_j(\pi) \right) \right)^{-1}.
    \label{probPVdef}
\end{equation}
Equivalently, PV can be characterized recursively, for every $i \in D$, as:
\begin{equation}
     \PV_i = \frac{R(D,v)}{R(D_{-i}, v)}
     \label{eq:def_ratiopot_PV}
 \end{equation}
where, for all $A \in {\cal P}(D)$, $\displaystyle R(A,v) = v(A) \left(\sum_{j\in A} R(A_{-j}, v)^{-1} \right)^{-1}$, and $R(\emptyset, v) = 1$. This recursive definition leads to the following identification \cite{Feldman2007}:
\begin{equation}   
\label{eq:prop_values1}
\PV_i =\cfrac{\sum_{\pi \in  \mathcal{S}_{D_{-i}}} \prod_{j=1}^{d-1} v\left(C_{j}(\pi)\right)^{-1}}{\sum_{\sigma \in  \mathcal{S}_{ D }} \prod_{j=1}^{d} v\left(C_{j}(\sigma)\right)^{-1}}.
\end{equation}
\end{dfi}
\begin{rmk}
The recursive function $R(D,v)$ defined in Definition~\ref{def:propVal} is better known as a \emph{ratio potential} in the cooperative game theory literature \cite{Feldman1999}, which is central to define certain allocations. This formulation is especially useful for estimation purposes and when it comes to proving results related to the PV. However, for the sake of conciseness, this notion is considered to be out of the scope of the present work, and is not further discussed.
\end{rmk}
They can also be characterized axiomatically (see \cite{Feldman2007}), as the unique allocation $\phi(D,v)$ respecting the following two axioms:
\begin{itemize}
    \item \textbf{Efficiency}: $\sum_{i=1}^d \phi_i = v(D)$;
    \item \textbf{Equal proportional gains}: for all $A \in {\cal P}(D)$, and for all $i,j \in A$, $i \not = j$:
    $$\frac{\phi_i\bigl( (A,v) \bigr)}{\phi_j\bigl( (A,v) \bigr)} = \frac{\phi_i\bigl( (A_{-j}, v) \bigr)}{\phi_j\bigl( (A_{-i}, v) \bigr)}.$$
\end{itemize}
Since the PV of positive monotonic games are efficient and nonnegative, they allow for the same meaningful interpretation as the Shapley values, \ie as shares of $v(D)$. The equal proportional gains axiom allows to better interpret the redistribution dynamic of this particular allocation scheme. For any two different players $i$ and $j$, the ratio of their allocations in any subgame $(A, v)$ (for every $A \in \mathcal{P}(D)$ such that $i,~j \in A$) must be invariant to removing each player's contribution to the other's allocation. In other words, the magnitude of the ratios must be preserved, independently of the possible interaction between $i$ and $j$, within any coalition they can belong to. This implicitly entails that the allocation tends to favor the players proportionally to their (marginal) contributions to every possible coalitions in the redistribution process.

As a frame of comparison, the Shapley values can also be characterized as the unique efficient allocation respecting the following axiom (see \cite{Feldman2007}):
\begin{itemize}
    \item \textbf{Balanced contributions:} for all $A \in {\cal P}(D)$, and for all $i,j \in A$, $i \not = j$:
    $$\phi_{i}(A, v)-\phi_{i}(A_{-j}, v)= \phi_{j}(A, v)-\phi_{j}(A_{-i}, v).$$
\end{itemize}
This axiom entails that for any two different players $i$ and $j$, the difference in each allocation by removing the other player to any subgame $(A,v)$ such that $i,~j \in A$ must remain equal, for any $A \in \mathcal{P}(D)$. In other words, the difference in allocation of the two players induced by the removal of the other player must be equal, implicitly entailing a balanced redistribution process, where individual and coalitional contributions are favored equally.

\begin{rmk}\label{rmk:ShapVal_vs_PV}
In a nutshell, one can remark that the redistribution processes in both allocations (Shapley values \vs PV) are fundamentally different: the PV redistribution process is \emph{proportional} meanwhile the Shapley values are \emph{egalitarian}.
\end{rmk}

The different behaviors between PV and Shapley values can be illustrated by considering a two-player game, \ie $(D = \{1,2\})$. The allocation are given, for any $i \in D$, by
\begin{subequations}
    \begin{alignat}{3}
        \PV_i\Big((D,v)\Big)&= v(\{i\})+ \frac{v(\{i\})}{v(\{1\})+v(\{2\})}\Big(v(D)-v(\{1\})-v(\{2\})\Big)\\
        \mbox{Shap}_i\Big((D,v)\Big)& = v\big(\{i\}\big) + \frac{1}{2}\Big(v(D)-v(\{1\})-v(\{2\})\Big).
    \end{alignat}
\end{subequations}
For both PV and Shapley values, each player receives its individual contribution, plus a weighted share of the value surplus generated due to their cooperation. In the literature, this surplus is referred to as the \emph{Harsanyi dividend} of the coalition $\{1,2\}$ \cite{Harsanyi1963}. The main difference between both allocations is the fact that the Shapley values redistribute exactly half of this dividend to each player (\ie egalitarian way), while the PV redistributes them proportionally (\ie proportional way) to each player's individual contribution.

It is important to notice that this allocation is only well-defined for positively defined value functions $v$. However, as stated in Definition~\ref{def:sobolGame}, the value function of Sobol' cooperative game is inherently nonnegative. The following section presents a continuous extension of the PV to nonnegative games, enabling their use for variance-based GSA purposes.

\subsection{Extension of proportional values to nonnegative games}
The main contribution of the present work is to propose an adaptation of the PV to variance-based GSA purposes. However, PV are only well defined on cooperative games $(D,v)$ with positive value function $v$, and Sobol' cooperative games are inherently endowed with nonnegative value functions. However, by leveraging the work proposed by \cite{Feldman2002}, it is possible to define a continuous extension of the PV, allowing their definition for games containing coalitions which have null value. The following result builds upon this extension, and allows to extend the PV to monotonic cooperative games with nonnegative value functions.
\begin{thm}[PV extension to nonnegative games]
\label{thm:PVExtension}
Let $(D,v)$ be a nonnegative, monotonic cooperative game with value function $v : \mathcal{P}(D) \rightarrow \mathbb{R}^+$. Let $k_{\max}$ denote the cardinal of the largest null coalition such that: $$k_{\max} := \underset{A \in \mathcal{P}(D)}{\max} \left\{\lvert A\rvert : v(A) = 0\right\}. $$
For any $i \in D$, let $\mathcal{K}_{-i}$ denote the set of null coalitions of $D_{-i}$, having a cardinal equal to $k_{\max}$:
$$\mathcal{K}_{-i} := \left\{A \in \mathcal{P}(D_{-i}) : v(A) = 0~\text{and}~ |A| = k_{\max} \right\}$$
and let $\mathcal{K} = \mathcal{K}_{-\emptyset}$ be the set of null coalitions of $D$ having a cardinal equal to $k_{\max}$. Moreover, for any $A \in \mathcal{P}(D)$, and for any $B \subseteq \overline{A}$, let $v_A : \mathcal{P}\left(\overline{A}\right) \rightarrow \mathbb{R}^+$ be the function defined as:
$$v_A(B) = v(A \cup B).$$
The allocation $\PVext \bigl( (D,v) \bigr) = (\PVext_1, \dots, \PVext_d)^\top \in \mathbb{R^d}$, defined for any $i \in D$ as:
\begin{equation}
    \PVext_i = \begin{cases}
    0 & \text{if } \forall A \in \mathcal{K},~ i \in A,  \\
    \dfrac{\sum_{A \in \mathcal{K}_{-i}} R(D_{-i} \setminus A, v_A)^{-1}}{\sum_{A \in \mathcal{K}} R(D \setminus A, v_A)^{-1}} & \text{otherwise},
    \end{cases}
    \label{eq:PV_cont}
\end{equation}
is a continuous extension of the PV on the set of nonnegative monotonic cooperative games, \ie for any positive cooperative game $(D,v)$, one has that:
$$\PVext \bigl( (D,v)\bigr) = \PV \bigl( (D,v) \bigr). $$
\end{thm}
A detailed proof of this result can be found in Appendix~\ref{app:proofThm}. Additionally to extending the PV to nonnegative games, Theorem~\ref{thm:PVExtension} allows to clearly identify players who receives a zero allocation. More precisely, a player $i$ receives a zero allocation if it is part of every largest coalitions with null value.

\begin{rmk}\label{rmk_PVext}
In the rest of this paper, any mention to the PV refers to their extended version to nonnegative monotonic games (\ie $\PVext$), as the usual chosen value functions for variance-based GSA (\eg typically, the closed Sobol' indices $\Sclos$) can indeed be equal to zero.
\end{rmk}

\subsection{Proportional marginal effects and exogeneity detection}
From Theorem~\ref{thm:PVExtension}, one can see that the PV of a cooperative game and the PV of its dual are not equal, unlike the Shapley values. In the case of the PV of Sobol' cooperative games, it is important to notice that focusing on the dual is more relevant to better detect exogenous variables (\ie variables that are irrelevant to the studied model). Indeed, taking the dual of a Sobol' cooperative game, i.e., considering $\St$ as a value function instead of $\Sclos$, allows to better detect exogenous inputs, thanks to the following classical result, echoing the work in \cite{Hart2018}. 
\begin{lme}\label{lme:StExo}
    Let $X=(X_1, \dots, X_d)^\top$ be random inputs and $G \in L^2(P_X)$ denote a model. One has, $\forall A \subset D$, 
    \begin{equation*}
        \St_A = \frac{\E{\V{G(X) \mid X_{\overline{A}}}}}{\V{G(X)}} = 0  \quad \iff \quad  G(X) = \E{G(X) \mid X_{\overline{A}}} \text{ a.s.}
    \end{equation*}
\end{lme}
A proof of this result can be found in Appendix~\ref{app:proofThm}. In other words, for any subset of inputs $A~\subseteq~D$, whenever $\St_A = 0$ indicates that $G(X)$ can be expressed as a function only depending on $X_{\overline{A}}$, which is strongly related to the definition of exogenous inputs (i.e., Definition~\ref{def:exogInput}). On the other hand, choosing $\Sclos$ as a value function leads to the interpretation that $\E{G(X) \mid X_A}$ is constant a.s., which is not necessarily equivalent to exogeneity. Hence, in the goal of exogenous input detection, choosing $\St$ as a value function is more relevant, especially in the case of the proportional values of Sobol' cooperative games.

This leads to the proposed cooperative game theory-inspired GSA indices called \emph{proportional marginal effects} (\PME). They are defined as follows:
\begin{dfi}[Proportional marginal effects]\label{def:pme}
Let $X=(X_1, \dots, X_d)^\top$ be random inputs, and let $Y=G(X)$ be the random output of a model. The proportional marginal effects are the proportional values of the dual of the monotonic Sobol' cooperative game related to the model $G$. They are defined as:
$$\PME = \PV \bigl( (D, \St) \bigr) \in \mathbb{R}^d.$$
\end{dfi}
The PME are efficient and nonnegative, i.e., they allocate shares of the output's variance to every inputs, and hence remain meaningful in practice even when inputs are dependent. As explained in Section~\ref{sec:PVbetterShap}, they differ from the Shapley effects on the underlying redistribution principle. More importantly, they allow to detect exogenous inputs to a model, thanks to the following result.

\begin{propo}\label{prop:EquivExoPME}
Let $X=(X_1, \dots, X_d)^\top$ be random inputs, let $G \in L^2(P_X)$ be a model, and let $E \subseteq D$. If $X_E$ is the largest $L^2$-exogenous subset inputs to $G$ among $X$, then 

$$\PME_i=0, \forall i \in E, \text{ and } \PME_j >0, \forall j \in \overline{E}.$$
\end{propo}

A proof of this result can be found in Appendix~\ref{app:proofThm}. In addition to offer a tool for factor prioritization, the \PME\, allow to detect exogenous inputs by granting them a null allocation. As explained in Section~\ref{sec:exoDetection}, they allow to circumvent one of the main practical drawback of the Shapley effects by enabling to detect exogenous inputs, while maintaining their strengths, i.e., a meaningful decomposition of the variance of the output of a numerical model when inputs are dependent. In the following, in order to better understand these novel sensitivity indices, their behavior is studied analytically on toy-cases, and compared to the Shapley effects, in the following section.

\section{Illustration on analytical cases}\label{sec:3}

Three analytical models are studied. The first one aims at illustrating the exclusion property of the $\PME$ ensured by Proposition~\ref{prop:EquivExoPME}, which means that $\PME$ avoid the Shapley's joke as illustrated in Example~\ref{exa:ShapleyJoke}. The second test-case is dedicated to highlighting the proportional principle of the $\PME$ (see Remark~\ref{rmk:ShapVal_vs_PV}), by studying their behavior w.r.t. the magnitude of the linear coefficient of one input in a correlated setting. Finally, the third model introduces a trade-off between individual and interaction effects between two inputs, and further highlights the difference in repartition between the Shapley effects and the proposed $\PME$.



\subsection{A linear model with an exogenous input}\label{testcase_linmod_exo}

This first analytical case aims at highlighting the exclusion property of the $\PME$, as ensured by Proposition~\ref{prop:EquivExoPME} and motivated by Example~\ref{exa:ShapleyJoke}. This first model reads:
\begin{equation}
    Y =G(X)= X_1 + X_2, \quad X=\Pmatrix{X_1 \\ X_2 \\ X_3} \sim \cal{N}\left(\Pmatrix{0\\0\\0}, \Pmatrix{1 &0 &\rho \\ 0 & 1 & 0 \\ \rho & 0 & 1} \right)
\end{equation}
where $-1<\rho<1$. One can notice that $X_3$ is, on purpose, exogenous as it does not intervene explicitly in the computation of $Y$. It is, however, linearly correlated to $X_1$. In such a case, traditional GSA tools such as first- and second-order Sobol' indices fail to be interpreted as shares of variance \cite{davgam21}.
In order to circumvent this problem, one can resort to compute the Shapley effects and PME of the inputs. For this test-case, reference analytical values for both indices are given in \mytabref{table:sh_PME_analytical_exo}.

\begin{table}[ht!]\label{table:sh_PME_analytical_exo}
\centering
\begin{tabular}{l|l}\toprule
$\Sh_1 = 1/2 - \rho^2/4$ & $\PME_1 = 1/2$\\
\midrule
$\Sh_2 = 1/2$ & $\PME_2 = 1/2$\\
\midrule
$\Sh_3 = \rho^2/4$ & $\PME_3 = 0$ \\
\bottomrule
\end{tabular}
\caption{Reference analytical values for Shapley effects and PME (toy-case~\ref{testcase_linmod_exo}).}
\end{table}



From the results provided in \mytabref{table:sh_PME_analytical_exo}, one can first notice that $X_3$ can receive a non-zero Shapley effect, dependent on the value of $\rho$. In highly correlated settings, $X_3$ can be considered to be almost as important as $X_1$. This behavior can be meaningful in practice, since $X_3$ is correlated with $X_1$, which is not an exogenous input. However, this interpretation is based on the knowledge of the underlying model, which is supposed to be black-box. As is, relying only on the Shapley effects, the practitioner would not be able to determine the exogenous nature of the inputs. If the aim of the sensitivity study is focused on better understanding the relationship between the model $G(.)$ and its inputs, somewhat independently from their probabilistic structure, the Shapley effects are hence not suitable alone.

One can notice that the PME does indeed detect $X_3$ as being an exogenous input, by granting it a zero allocation. Moreover, in this setting, the PME is not influenced by the linear correlation $\rho$ between $X_1$ and $X_3$. In combination with the Shapley effects, additional insights on $G$ can be extracted from the initial study: while $X_3$ can have an effect on $G$ through its correlation with other inputs, it is exogenous to $G$. Additionally, by allocating half the output's variance to both $X_1$ and $X_2$, the PME also indicates an equal influence. Hence, by combining the interpretation of both indices, one can interpret these results as follows: $X_3$ is an exogenous variable (PME), but it bears an effect on $G$ through its dependence with other inputs (Shapley effects) and, moreover, $X_1$ and $X_2$ seem to bear an equal influence on the output's variance, whenever $X_3$ is detected as exogenous (PME).

It is important to note that neither $G$ nor the dependence structure between the inputs need to be known for this interpretation. Hence, both indices are complementary and allow for a more precise interpretation of the studied model and its interaction with the inputs and their probabilistic structure.




\subsection{Unbalanced linear model}\label{testcase_linmod_unbalanced}


Echoing Remark~\ref{rmk:ShapVal_vs_PV}, beyond the detection of exogenous inputs, the Shapley effects and the PME fundamentally differ on their redistribution process. While the Shapley effects allocate importance in an egalitarian fashion, the PME follows a proportional principle. This toy-case aims at highlighting this difference, by introducing a coefficient in a linear model with three correlated Gaussian inputs. This use-case is referred to as \emph{unbalanced} since the three linear coefficient are different. This toy-case writes:
\begin{subequations}
    \begin{alignat}{3}
        Y &= G(X) =  X_1 +\beta X_2+ X_3, \quad X=\Pmatrix{X_1 \\ X_2 \\ X_3} \sim \cal{N}\left(\Pmatrix{0\\0\\0}, \Pmatrix{1 &0 &0\\ 0 & 1 & \rho \\ 0 & \rho & 1} \right),\label{eq_model_unbalanced}\\
        \V{Y} &= 2+ \beta^2 +2\rho\beta\label{eq_varY_unbalanced}.
    \end{alignat}
\end{subequations}
The analytical shares of output variance, according to the Shapley effects and the PME are given in \mytabref{table:sh_PME_analytical_unbalanced}.
\begin{table}[ht!]\label{table:sh_PME_analytical_unbalanced}
\centering
\begin{tabular}{l|l}\toprule
$\V{Y} \times \Sh_1 = 1$ & $\V{Y} \times \PME_1 = 1$\\
\midrule
$\V{Y} \times \Sh_2 = \beta^2+\beta\rho + \frac{1}{2} \rho^2(1-\beta^2)$ &
$\V{Y} \times \PME_2 = \frac{\beta^2(1+\beta^2+2\rho\beta)}{(1+\beta^2)}$\\
\midrule
$\V{Y} \times \Sh_3 = 1+\beta\rho - \frac{1}{2} \rho^2(1-\beta^2)$ &
$\V{Y} \times \PME_3 = \frac{(1+\beta^2+2\rho\beta)}{(1+\beta^2)}$ \\
\bottomrule
\end{tabular}
\caption{Reference analytical values for Shapley effects and PME (toy-case~\ref{testcase_linmod_unbalanced}).}
\end{table}

One can notice that, by considering the \emph{balanced} case (i.e., $\beta=1$), the Shapley effects and PME are equal. However, as soon as the model is unbalanced, one can notice that both allocations behave in a completely different fashion as soon as $\rho$ approaches $1$. Using an asymptotic-analysis-based reasoning, one can obtain the following set of resulting approximation:


\begin{subequations}
    \begin{alignat}{3}
        \Sh_1 = \text{PME}_1 & \xrightarrow[\rho \to 1]{}~~~~~~~~~ \frac{1}{2 + \beta^2 + 2\beta }& \quad \xrightarrow[\beta \to \infty]{} 0 ,\\
        \Sh_2,\Sh_3 &  \xrightarrow[\rho \to 1]{}~~~~~~~~\frac{\frac{1}{2}\beta^2 + \beta + \frac{1}{2} }{2 + \beta^2 + 2\beta } & \quad \xrightarrow[\beta \to \infty]{} \frac{1}{2} ,\\
        \PME_2 & \xrightarrow[\rho \to 1]{}~~~\frac{\beta^2(1 + \beta^2 + 2\beta)  }{(2 + \beta^2 + 2\beta)(1 + \beta^2)} & \quad \xrightarrow[\beta \to \infty]{} 1 ,\\
        \PME_3 & \xrightarrow[\rho \to 1]{}~~~\frac{(1 + \beta^2 + 2\beta)}{(2 + \beta^2 + 2\beta)(1 + \beta^2)} & \quad \xrightarrow[\beta \to \infty]{} 0.
    \end{alignat}
\end{subequations}
In other words, in extreme cases of positive linear correlation between $X_2$ and $X_3$, the Shapley effects allocates half the importance to each input despite a fairly high $\beta$ value in favor of $X_2$. The PME, on the other hand, tend to favor $X_2$ by granting it the whole variance, despite the high correlation with $X_3$. This behavior highlights the ``egalitarian \vs proportional'' behavior of both types of effects: the Shapley effects tend to consider $X_2$ and $X_3$ as equally important due to their high correlation, while the PME favor $X_3$ in regards of its high linear coefficient.

While these results inform on the asymptotic behavior of both indices, their difference can also be highlighted for punctual values of $\rho$ and $\beta$. \myfigref{fig:anal_ex_2_1} illustrates the behavior of both indices \wrt $\rho$, for two different values of $\beta$ (namely, $2$ and $10$).
Whenever $\beta=2$, one can notice that $\PME_2$ increases w.r.t. $\rho$, while $\Sh_2$ decreases after $\rho \simeq -0.24$, and both indices are concave w.r.t. $\rho$. On the other hand, $\Sh_3$ is convex w.r.t. $\rho$ and becomes increasing at $\rho \simeq -0.54$, while $\PME_3$ remains concave increasing. At extreme values of $\rho$ (i.e., close to $-1$ or to $1$), one can notice that $\Sh_2$ and $\Sh_3$ are considered equally important. Furthermore, one can notice that $\PME_2>\PME_3$, whatever the magnitude of their correlation. Increasing $\beta$ to $10$ exacerbates this behavior of the Shapley effects. However, the $\PME$ behave differently: $X_1$ and $X_3$ are given a negligible part of variance, while $X_2$ is granted a seemingly constant share, \wrt $\rho$, hovering around $98\%$. 

In conclusion, in this unbalanced case, the proportional redistribution property of the $\PME$ allows for a clearer importance hierarchy, even in situation of extreme correlation. On the other hand, the Shapley effect tends to the even importance out between the correlated inputs, leading to a potentially indecisive importance hierarchy.



\begin{figure}[!ht]
    \centering
    \includegraphics[width=\textwidth, trim= 0cm 6cm 0cm 0cm]{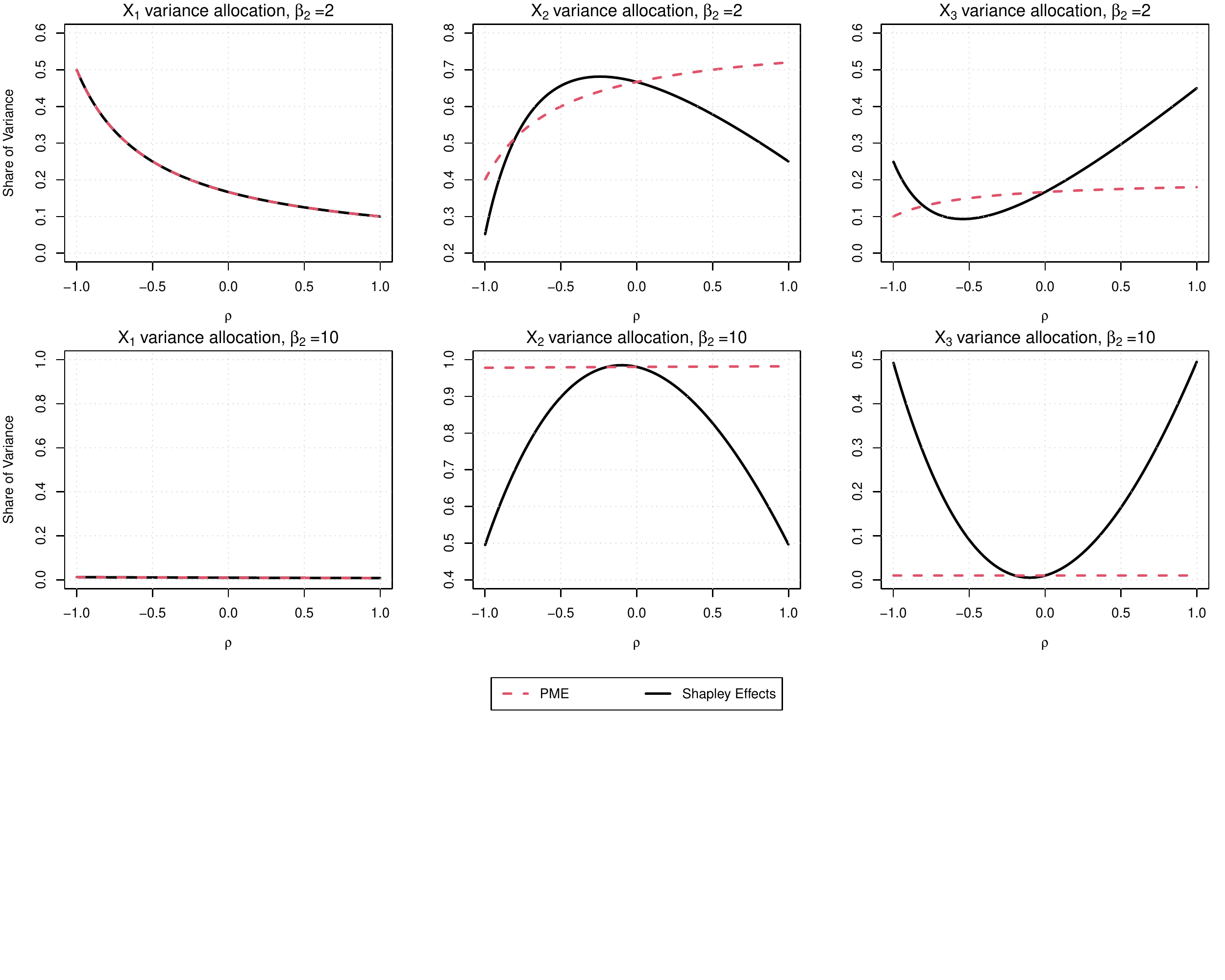}
    \caption{PME and Shapley effects \wrt $\rho$ (toy-case~\ref{testcase_linmod_unbalanced}). Top row depicts the allocations for $\beta=2$ while the bottom row is for $\beta=10$.}
    \label{fig:anal_ex_2_1}
\end{figure}




\subsection{Unbalanced linear model with interactions}\label{testcase_linmod_interactions}
This test-case aims at studying and comparing the behavior in a trade-off between individual and interaction effects. A parameter $\alpha \in [0,1]$ is introduced as part of a linear model comprising an interaction term. This particular unbalanced linear model is given as:
\begin{subequations}
    \begin{alignat}{3}
        Y &= G(X) =  X_1 + (1-\alpha)X_2 +  \alpha X_1X_2, \quad X=\Pmatrix{X_1 \\ X_2} \sim \cal{N}\left(\Pmatrix{0\\0}, \Pmatrix{1 & \rho \\ \rho & 1 } \right),\label{eq_model_interactions}\\
        \V{Y} &= 2 + (1-\alpha)^2 + 2(1-\alpha)\rho +\rho^2\label{eq_varY_interactions}.
    \end{alignat}
\end{subequations}

In other words, $\alpha$ controls the trade-off between the individual effect of $X_2$ and its interaction term with $X_1$. When $\alpha =0$, the model becomes balanced between $X_1$ and $X_2$, and when $\alpha=1$, $X_2$ only interacts with $X_1$. In both cases, $X_1$ remains present in the model. Analytical formulas for the Shapley effects and the PME are given in \mytabref{table:sh_PME_analytical_interactions}.

\begin{table}[ht!]\label{table:sh_PME_analytical_interactions}
\centering
\begin{tabular}{l|l}\toprule
$2\V{Y} \times \Sh_1 = 3 + \rho^2(1-\alpha)^2  + 2\rho(1-\alpha)$ & $\PME_1 = \frac{2}{3+(1-\alpha)^2}$\\
\midrule
$2\V{Y} \times \Sh_2 = 1+2\rho^2+(2-\rho^2)(1-\alpha)^2 +2\rho(1-\alpha)$ & $\PME_2 = \frac{(1-\alpha)^2+ 1}{ 3+(1-\alpha)^2}$\\
\bottomrule
\end{tabular}
\caption{Reference analytical values for Shapley effects and PME (toy-case~\ref{testcase_linmod_interactions}).}
\end{table}

When $\alpha = 0$ the model is balanced, and the Shapley effects and the PME are equal, allocating half of the variance to each input.
To better illustrate the redistribution principles \wrt both correlation and interaction, $(\alpha,\rho)$-plane plots are provided, for each effect in \myfigref{fig:anal_ex_3_1}.
One can notice that $\Sh_1$ and $\PME_1$ are increasing \wrt $\alpha$, but $\Sh_2$ and $\PME_2$ show a decreasing behavior. However, the Shapley effects effectively depend on the correlation coefficient $\rho$ in addition to $\alpha$, while the $\PME$ only depend on $\alpha$.



\begin{figure}[!ht]
    \centering
    \includegraphics[width=0.495\textwidth]{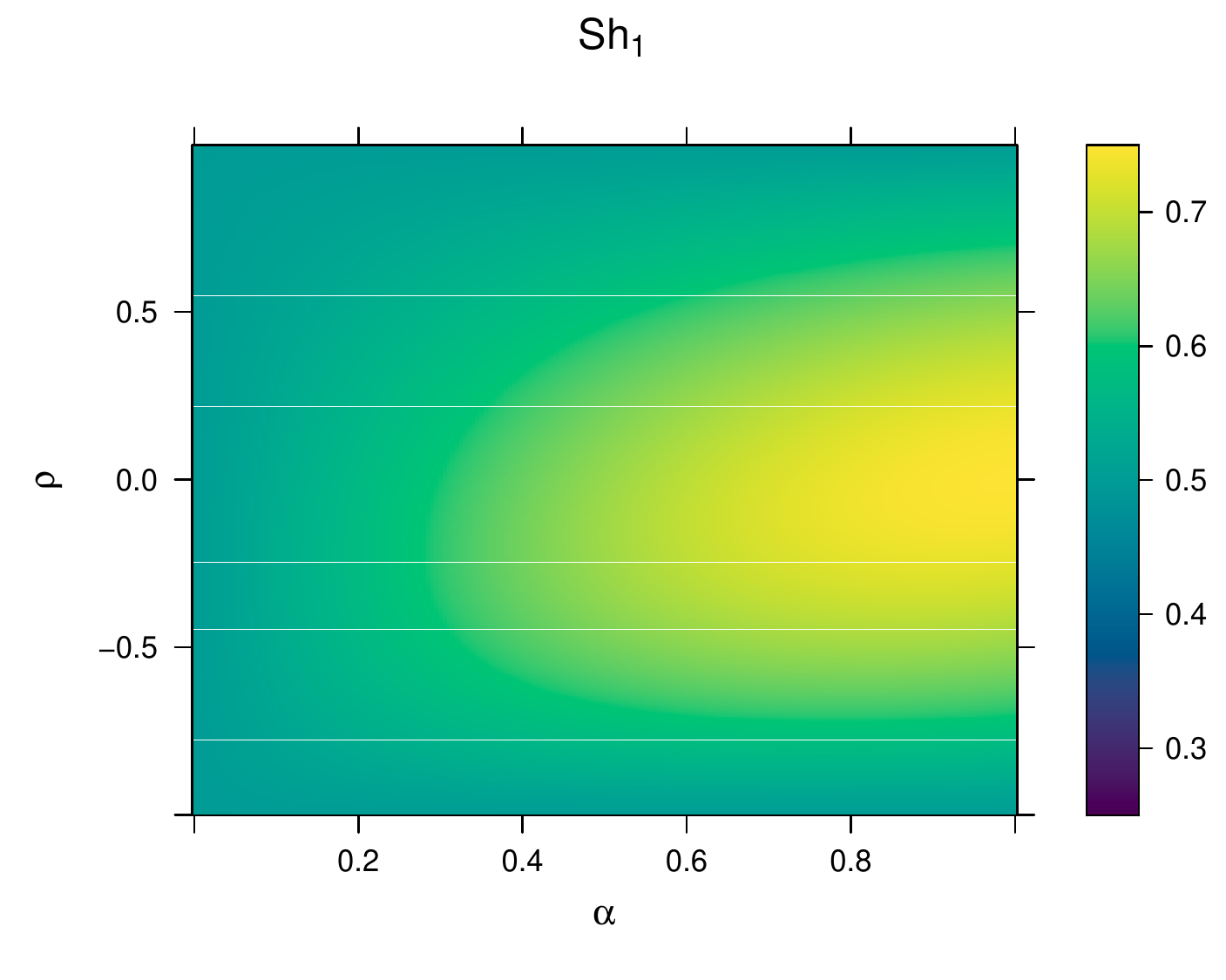} \includegraphics[width=0.495\textwidth]{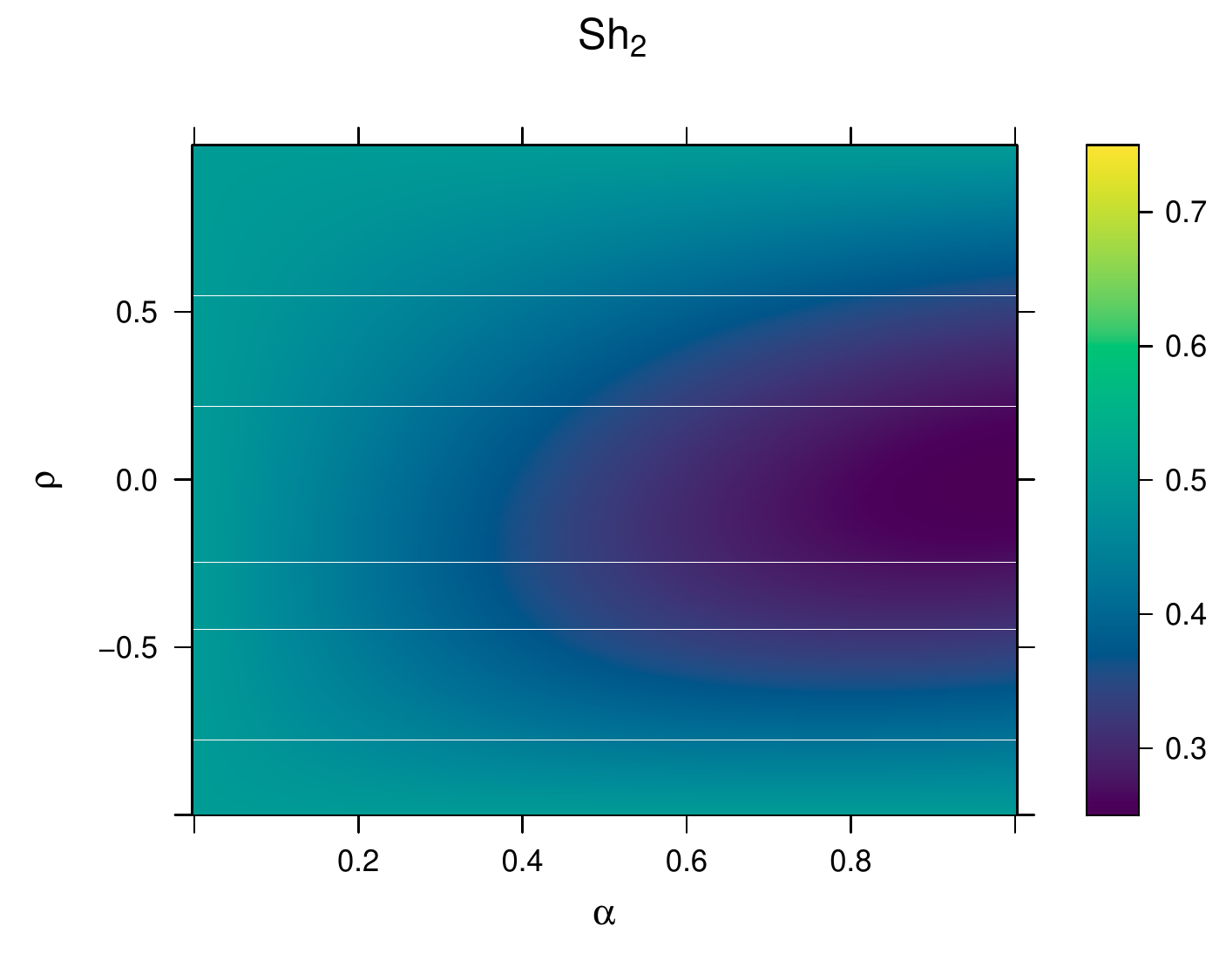}
    \includegraphics[width=0.495\textwidth, trim=0cm 2cm 0cm 0cm]{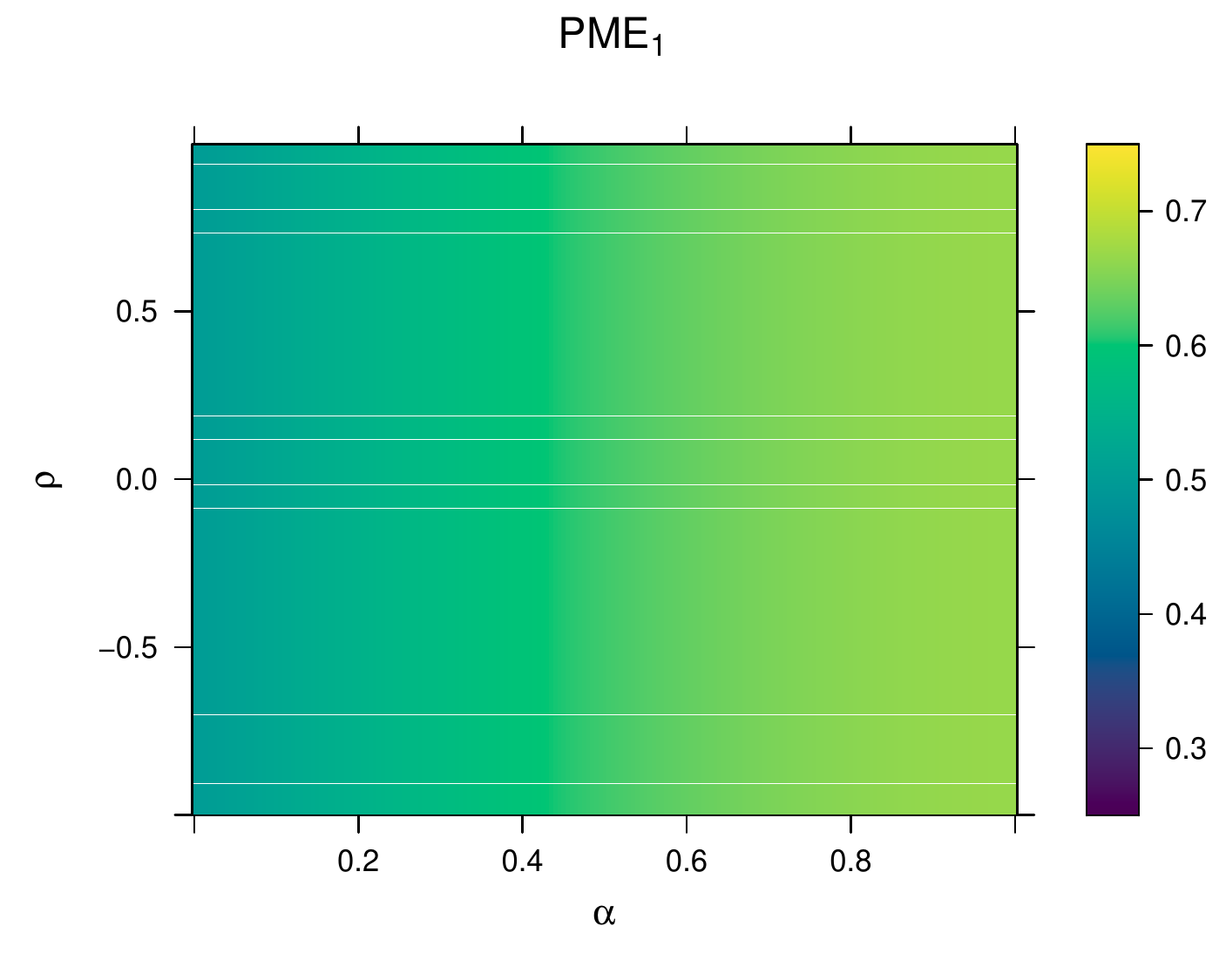} \includegraphics[width=0.495\textwidth, trim=0cm 2cm 0cm 0cm]{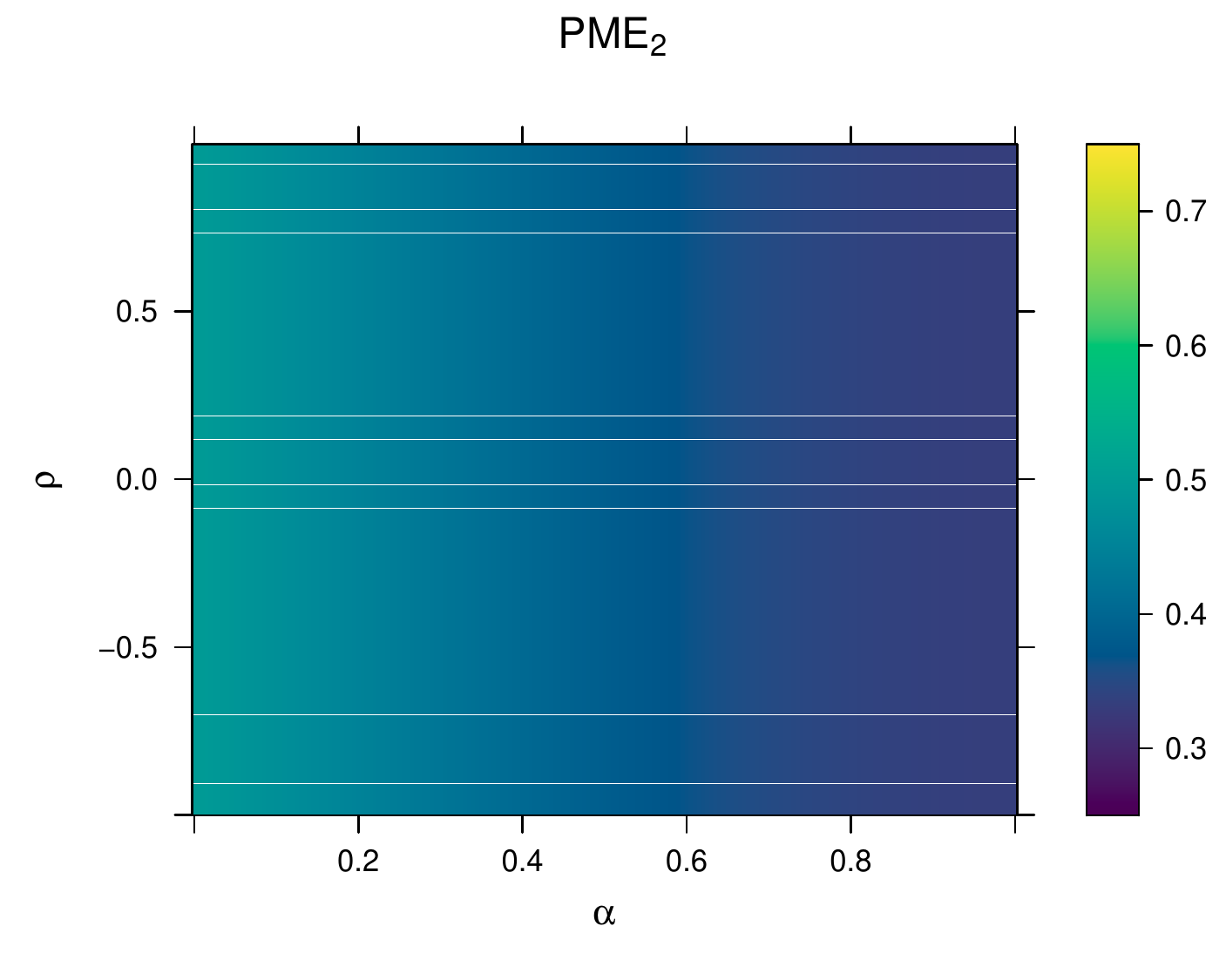}
    \caption{PME and Shapley effects in the $(\alpha, \rho)$-plane for test-case~\ref{testcase_linmod_interactions}.}
    \label{fig:anal_ex_3_1}
\end{figure}

Focusing on the behavior of both effects \wrt the interaction, one can first focus on the $\alpha$-axis of the plots in \myfigref{fig:anal_ex_3_1}. Whenever $\alpha$ is close to 0, one can notice that both indices tend to allocate an equal share of the output variance to both inputs. As $\alpha$ increases, the $\PME$ grants an increasing share of the output variance to $X_1$, from half to two thirds of the variance, and a decreasing share to $X_2$, from half to one third. As explained, the $\PME$ does not depend on the correlation between both inputs in this case.

Focusing on the Shapley effects, low values of $\alpha$ seem to share equally the variance. However, the higher $\alpha$ becomes, the more the sharing mechanism depends on $\rho$. When $\rho$ is between $-0.5$ and $0.5$, and $\alpha$ is close to $1$, $\Sh_1$ increases, with a maximum allocation of $0.75$ taken at $(\rho=0, \alpha=1)$, while $\Sh_2$ decreases, with a minimum allocation of $0.25$ at the same point. The redistribution proposed by the Shapley values is sensible in this case, since that when $\alpha=1$, $X_1$ is endowed with its individual effect, and an interaction effect with $X_2$, while $X_2$ only interacts with $X_1$ in this model. However, this interpretation is subject to a moderate level of correlation. As soon as the correlation rises, for values of $\rho$ close to $1$ and $-1$, the Shapley effects still grant half the output variance to both inputs, whatever the value of $\alpha$ is.

In this toy-case, the interpretation of the $\PME$ is more robust to high level of correlation if the goal of the study is to gather insights on the intricacies of the model $G$. On the other hand, the Shapley effects tend to even the importance between the correlated inputs whenever their correlation levels are fairly high, resulting in a potentially indecisive importance ranking.


\subsection{First conclusions}
From these few tests of linear models with a small number of correlated Gaussian inputs, and with the inclusion of simple interaction, the following conclusions can be drawn:
\begin{itemize}
    \item whenever the inputs are correlated, the Shapley effect does not detect exogenous inputs, while the $\PME$ does;
    \item in the highly correlated cases, the Shapley effect can lead to indecisive importance ranking, while the $\PME$ allows for a more pronounced hierarchy in accordance with the model $G$;
    \item overall, the $\PME$ is less sensitive to correlations, while the Shapley effect can vary greatly.
\end{itemize}
However, these conclusion are subject to the very specific presented toy-cases. The following section presents more ambitious models, closer to real-world applications, where both effects are estimated instead of computed analytically.

\section{Estimation and numerical results}\label{sec:4}
In this section, estimation schemes of the $\PME$ are presented: they rely on the same ingredients as the Shapley effects. Then, two numerical cases are studied using these schemes: a modified Ishigami function and a robot arm model. 

\subsection{Estimation strategies}\label{sec:estimation}

The plug-in estimation of the $\PME$ relies on the exact same elements than the estimation of the Shapley effects. For the sake of completeness, the classical estimation framework is briefly stated. Following the two-steps methodology presented in \cite{BrotoPhD}, initially developed for Shapley effects' estimation, one can estimate the PME in two distinct steps:
\begin{itemize}
    \item Step 1: Estimate the \emph{conditional elements}, \ie $\St_A$, $\forall A \in \cal{P}(D)$;
    \item Step 2: Perform an \emph{aggregation procedure} via a direct plug-in of the estimated conditional elements in \myeqref{eq:PV_cont}.
\end{itemize}
Only the aggregation procedure differs between the estimation of the $\PME$ and the Shapley effects. This entails that the estimation cost in terms of model evaluations is exactly the same for the $\PME$ than for the Shapley effects. Furthermore, both indices can be evaluated ``at-once'', by using the same conditional elements estimates. Two estimation schemes concerning the conditional elements are presented, followed by a break-down of the algorithmic logic behind the specific aggregation procedure of the $\PME$.

A first estimation scheme for the total Sobol' indices relies on a double Monte Carlo procedure. It can be found in \cite{sonnel16}. This estimator requires the ability to randomly sample from every possible conditional distributions of the conditional random variables $X_A|X_{\overline{A}}$, from every marginal distributions, \ie to simulate \iid observations of $X_A$, for all $A \in \cal{P}(D)$, as well as from the joint distribution of the random variable $X$. For any $A \subseteq D$, the required \iid samples to estimate $\St_A$ are the following:
\begin{itemize}
    \item an \iid sample of size $N_v$ of $X$, denoted by $(X^{(1)}, \dots ,X^{(N_v)})$;
    \item another \iid sample of size $N_o$ drawn from $X_{\overline{A}}$ and denoted by $(X_{\overline{A}}^{(1)}, \dots, X_{\overline{A}}^{(N_o)})$;
    \item for each element $X_{\overline{A}}^{(i)}, i=1, \dots, N_o$, a corresponding sample of size $N_i$ drawn from $X_{A}|X_{\overline{A}}=X_{\overline{A}}^{(i)}$ denoted by $(X_{A|\overline{A}, i}^{(1)}, \dots, X_{A|\overline{A},i}^{(N_i)})$. Additionally, for $i=1,\dots, N_o$, $j=1,\dots, N_i$, one conditional observation of the input $k \in D$ is given by
    $$\widetilde{X}_{(\overline{A}, i), k}^{(j)} = \begin{cases} 
    X_{\overline{A}}^{(i)} & \text{if } k \in \overline{A}, \\
    X_{A|\overline{A}, i}^{(j)} & \text{otherwise}.
    \end{cases} $$
    Subsequently denote, for $j=1,\dots, N_i$, a full conditional observation w.r.t. $X_{-A}^{(i)}$:
    $$\widetilde{X}_{\overline{A}, i}^{(j)} = \left( \widetilde{X}_{(\overline{A}, i), k}^{(j)}\right)_{k =1, \dots, d} \in \mathbb{R}^d.$$
    In other words, the $(N_i \times d)$ bloc-matrix composed of every full conditional observation, i.e.,
    $$\widetilde{X}_{\overline{A}, i} = \begin{pmatrix} \widetilde{X}_{\overline{A}, i}^{(1)} \\ \vdots \\ \widetilde{X}_{\overline{A}, i}^{(N_i)} \end{pmatrix}$$
    is a concatenation of columns of conditionally simulated samples (i.e., simulations from $X_{A}|X_{\overline{A}}=X_{\overline{A}}^{(i)}$) and columns composed of the values by which the distribution is conditioned on (i.e., $X_{\overline{A}}^{(i)}$).
    \end{itemize}
    
Let $A \subseteq D$. One can first estimate $\V{G(X)}$ using the unbiased classical variance estimator:
$$\widehat{\mathbb{V}(G(X))} = \frac{1}{N_v -1}\sum_{i=1}^{N_v} \left(G\left(X^{(i)}\right) - \overline{G(X)}\right)^2 $$
with $\overline{G(X)} = \frac{1}{N_v} \sum_{i=1}^n G(X^{(i)})$.
Then, for every observation $i=1,\dots, N_o$ of the sample of $X_{\overline{A}}$, the conditional variance $\V{G(X)~\mid~X_{\overline{A}}=X_{\overline{A}}^{(i)}}$ is estimated by
$$\widehat{V}_{\overline{A}}^{(i)} := \frac{1}{N_i -1}\sum_{j=1}^{N_i} \left( G\left(\widetilde{X}_{\overline{A}, i}^{(j)}\right) - \frac{1}{N_i} \sum_{k=1}^{N_i} G\left(\widetilde{X}_{\overline{A}, i}^{(k)}\right)\right)^2. $$
Finally, a consistent estimate of $\St_A$ is given by
$$\widehat{\St_A} = \left(\frac{1}{N_o} \sum_{i=1}^{N_o} \widehat{V}_{\overline{A}}^{(i)}\right) \bigg/ \widehat{\mathbb{V}(G(X))}.$$

The ability to sample from these conditional distributions can be difficult in practice, especially in cases where only an \iid sample of the input/output is available. A given-data estimation scheme for the conditional elements has then been proposed in the literature. It relies on the approximation of the conditional samples using a nearest-neighbor scheme, and thus allows to estimate every $\St_A$ with only one observed \iid sample of $X$ and its corresponding model output. For the sake of conciseness, this methodology is not detailed in the present paper, but is used in the use-case presented in Section~\ref{sec:robot}. One can refer to \cite{brobac20, ilicha21} for additional theoretical and computational details on this estimation method.

Given an estimation $\widehat{\St_A}$ for every conditional element $A \subseteq D$, the aggregation procedure leading to the PME can be computed using its recursive definition (see, Eq.~(\ref{eq:PV_cont})). It relies on the computation of the ratio potential, i.e., the function $R$ in Eq.~(\ref{eq:def_ratiopot_PV}), as described in Algorithm~\ref{alg:R_recur}.

\begin{algorithm}[!ht]
\caption{Computation of the ratio potential $R(A, v)$.}
\begin{algorithmic}[1]
\Require $A \subseteq D$, $v$.
\State $a \leftarrow |A|$
\State $P^- \leftarrow (v(\{1\}), \dots, v(\{ a\})) \in \mathbb{R}^d_+$
\For{$k=2, \dots, a-1$} 
    \State $m \leftarrow {a \choose k}$
    \State $P \in \mathbb{R}^m$
    \For{$B \subset A$ s.t. $|B|=k$} (possibly in parallel)
        \State $\mathcal{J} \leftarrow \{C \subset B \mid \forall j \in B, C=B\setminus\{j\} \}$
        \State $P_B = v(B) \left( \sum_{C \in \mathcal{J}} (P^{-}_{C})^{-1} \right)^{-1} $
    \EndFor
    \State $P^- \leftarrow P$
\EndFor
\State \Return $R(A, v) \leftarrow v(A)\left( \sum_{j=1}^m (P_j)^{-1}\right)^{-1}$
\end{algorithmic}
\label{alg:R_recur}
\end{algorithm}

With the ability to compute the ratio potential for any subset $A$ and any function $v$, one can proceed to compute the PME, as detailed in Algorithm~\ref{alg:PME_comp}.

\begin{algorithm}[!ht]
\caption{PME plug-in estimation.}
\begin{algorithmic}[1]
\Require $D$, $\widehat{\St_A}, \forall A \in \mathcal{P}_D$.
\State $d \leftarrow |D|$
\State $\text{PME} \in \mathbb{R}^d$
\State $k_{\max} = \underset{A \subseteq D}{\max} \{|A| : \widehat{\St_A}=0 \} $
\State $\mathcal{K} = \{A \subseteq D : |A|=k_{\max}, \widehat{\St_A}=0 \}$
\State $E = \{i \in D : \forall A \in \mathcal{K}, i \in A \}$
\State $\widetilde{\St_B}(.) \leftarrow \widehat{\St}_{B \cup .}$
\For{$i \in E$}
    \State $\text{PME}_i = 0$
\EndFor
\State $R_{\mathcal{K}} <- \sum_{A \in \mathcal{K}} \left(R\left(D \setminus A, \widetilde{\St}_A(.) \right) \right)^{-1}$
\For{$j \in D \setminus E$} (possibly in parallel)
    \State $\mathcal{K}_j \leftarrow \{A \subseteq D_{-j} \mid |A|=k_{\max}, v(A) = 0 \}$
    \State $\text{PME}_j \leftarrow \sum_{A \in \mathcal{K}_j} \left(R\left(D_{-j} \setminus A, \widetilde{\St_A}(.) \right)\right)^{-1} \big/ R_{\mathcal{K}}$
\EndFor
\State \Return $\text{PME}$
\end{algorithmic}
\label{alg:PME_comp}
\end{algorithm}

\subsection{Ishigami model with a correlated exogenous input}
In order to further study the behavior of the $\PME$, the Ishigami model, well-known in GSA (see, e.g., \cite{davgam21}), is first considered.
The Ishigami model is given by
$$G(X) = \sin(X_1) + 7 \sin^2(X_2) + 0.1 X_3^4 \sin(X_1).$$
In our study, the  following probabilistic structure of the inputs is considered:
$$X = \begin{pmatrix}X_1 \\ \vdots \\ X_4 \end{pmatrix} \sim \mathcal{N}\left( \begin{pmatrix} 0 \\ \vdots \\ 0\end{pmatrix}, \begin{pmatrix}
(\pi/3)^2 & 0 & 0 & \rho \\
0 & (\pi/3)^2 & 0 &0 \\
0& 0 & (\pi/3)^2&0 \\
\rho & 0 & 0 & (\pi/3)^2
\end{pmatrix} \right). $$
One can notice that $X_4$ is, by design, an exogenous input, but it is linearly correlated to $X_1$ by means of the parameter $\rho \in (-1,1)$. The Shapley effects and the $\PME$ are estimated using a Monte Carlo procedure, as presented in Section~\ref{sec:estimation}, with chosen sample sizes $N_v = 10^5$, $N_o = 2\times 10^3$ and $N_i=300$, for various values of $\rho$ (from $-0.99$ to $0.99$ with a step of $0.01$). Each Monte Carlo estimation has been independently repeated $200$ times in order to obtain confidence intervals. The results are provided in Fig. \ref{fig:ishi_rho}.

\begin{figure}[!ht]
    \centering
    \includegraphics[width=\textwidth, trim=0cm 6cm 0cm 0cm]{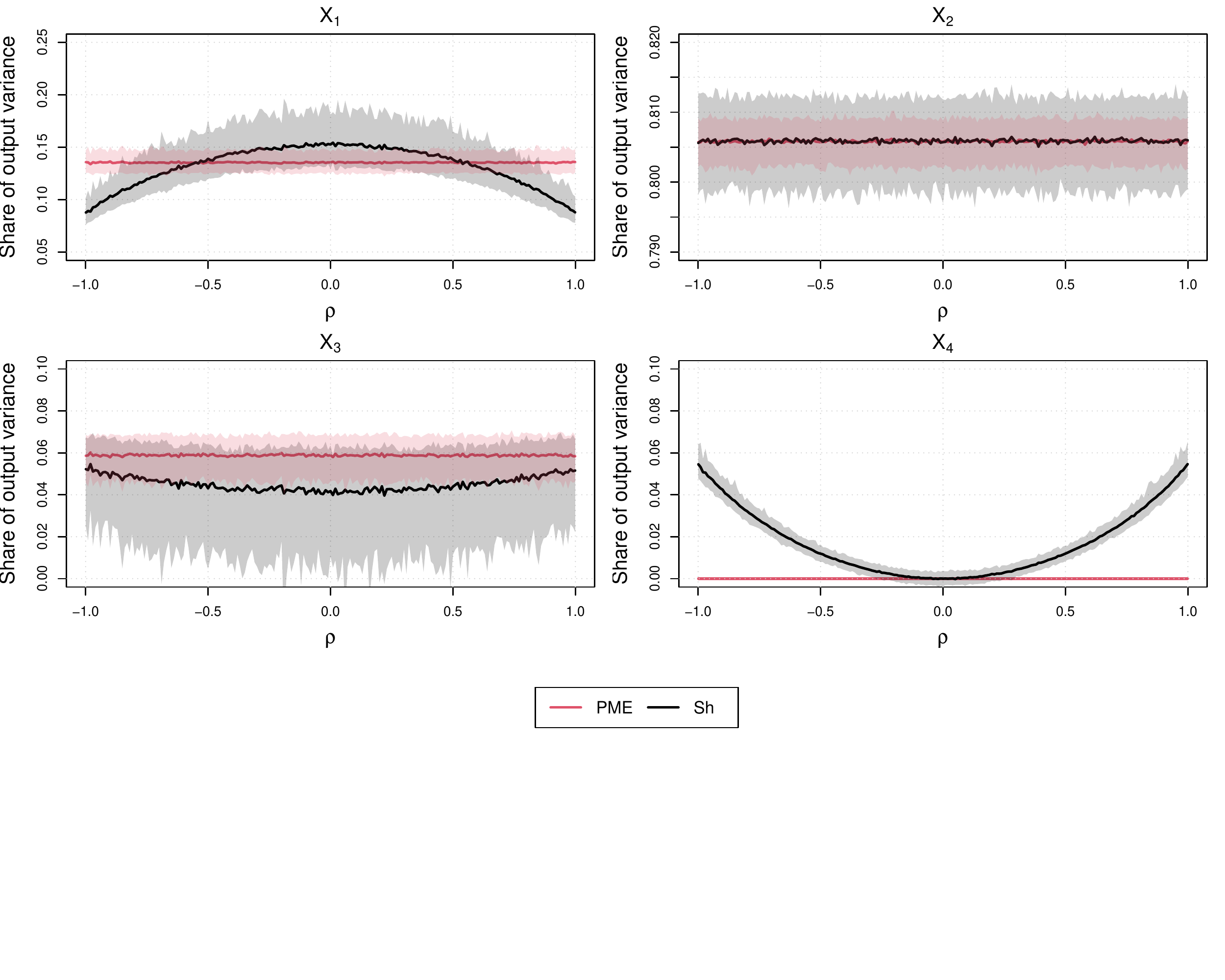}
    \caption{PME and Shapley effects for the Ishigami model with a spurious variable, with respect to the correlation coefficient $\rho$ between $X_1$ and $X_4$. The grey and red areas around the solid plots give the $95\%$-confidence intervals of the estimates.}
    \label{fig:ishi_rho}
\end{figure}

First, one can notice a strong influence of $X_2$, whose $\PME$ and Shapley effects are equal and constant along $\rho$. This result is expected, since $X_2$ has no interaction or correlation with other inputs in the Ishigami model, and hence its importance should not be subject to variation w.r.t. the correlation intensity. 
Second, focusing on $X_1$ and $X_4$, one can notice the same behavior of the Shapley effects as depicted previously. Despite the fact that $X_4$ is exogenous, in situation of extreme correlation, $\Sh_4$ can be as high as $\Sh_1$, which echoes the results in \cite{ioopri19}, but effectively grants a zero allocation to $X_4$ whenever both inputs are independent (i.e., $\rho=0$). However, their $\PME$ differ, in the sense that $\PME_1$ is constant w.r.t. $\rho$, while $\PME_4$ is equal to zero whatever the correlation value. Hence $X_4$ is effectively detected as being exogenous. 
Third, one can notice that $\Sh_3$ does vary w.r.t. $\rho$, which can be understood by the fact that $X_3$ interacts with $X_1$ in the model, which is itself correlated to $X_4$. However, since the $\PME$ detects $X_4$ as being exogenous, $\PME_3$ remains constant w.r.t the correlation structure. 
Finally, focusing on $X_3$ and $X_1$ whenever $\rho=0$ (i.e., the inputs are independent), one can notice that $\Sh_1 > \PME_1$, and $\Sh_3<\PME_3$. This can be understood as the expression of the proportional versus the egalitarian redistribution schemes. While the Shapley effects effectively grants half the interaction surplus to both inputs, the $\PME$ tend to favor $X_3$. This can be understood by the fact that $X_1$ does not have an overwhelmingly higher overall effect on $G$ than $X_3$.

Overall, the $\PME$ are less sensitive to the correlation of exogenous inputs than the Shapley effects. In conclusion, this toy-case highlights further the fact that both effects are complementary when it comes to a more precise interpretation of the model when its inputs are correlated. It reinforces the previously found behavioral tendencies in a less straightforward model.

\subsection{Robot arm model}\label{sec:robot}
In this use-case, one studies a model of the position (on the two-dimensional plane) of a robot arm with four segments \cite{anowe01}. The arm shoulder is fixed at the origin and the robot's segments have lengths $L_i$ ($i=1,\ldots,4$). Each segment is positioned at an angle $A_i$ ($i=1,\ldots,4$) with respect to the horizontal axis. While, in the original model, the inputs are assumed to be independent, statistical dependence is introduced here between the angles and between the segment lengths. The probabilistic structure on the inputs can be described as follows:
\begin{itemize}
    \item The angles $A_i$ ($i=1,\ldots,4$) follow a uniform distribution over $[0,2 \pi]$. They are pairwise correlated by the way of a Gaussian copula with correlation parameter equal to $0.95$;
    \item The lengths are sequentially built: $L_1$ follows a uniform distribution over $[0,1]$, while $L_i$ ($i=2,\ldots,4$) follows a uniform distribution over $[0,L_{i-1}]$. These inequality constraints create strong correlation between the lengths.
\end{itemize}
The model's output is the distance of the end of the robot arm to the origin, and writes:
$$Y = \left\{ \left[\sum_{i=1}^4 L_i \cos\left(\sum_{j=1}^i A_j\right)\right]^2 + \left[\sum_{i=1}^4 L_i \sin\left(\sum_{j=1}^i A_j\right)\right]^2 \right\}^{1/2}.$$
A unique \iid sample of size $2000$ of these $8$ inputs has been simulated, on which the output of the model has been computed. Figure \ref{fig:scatrobot} illustrates this data sample by the way of the pairwise scatter-plots, the marginal distributions of each input by means of histograms, and the dependence structure with estimated correlation coefficients. One can also notice first-order tendencies of the different inputs on the output $Y$ (last row).

\begin{figure}[ht!]
    \centering
    \includegraphics[width=\textwidth, trim=0cm 2cm 0cm 1.5cm]{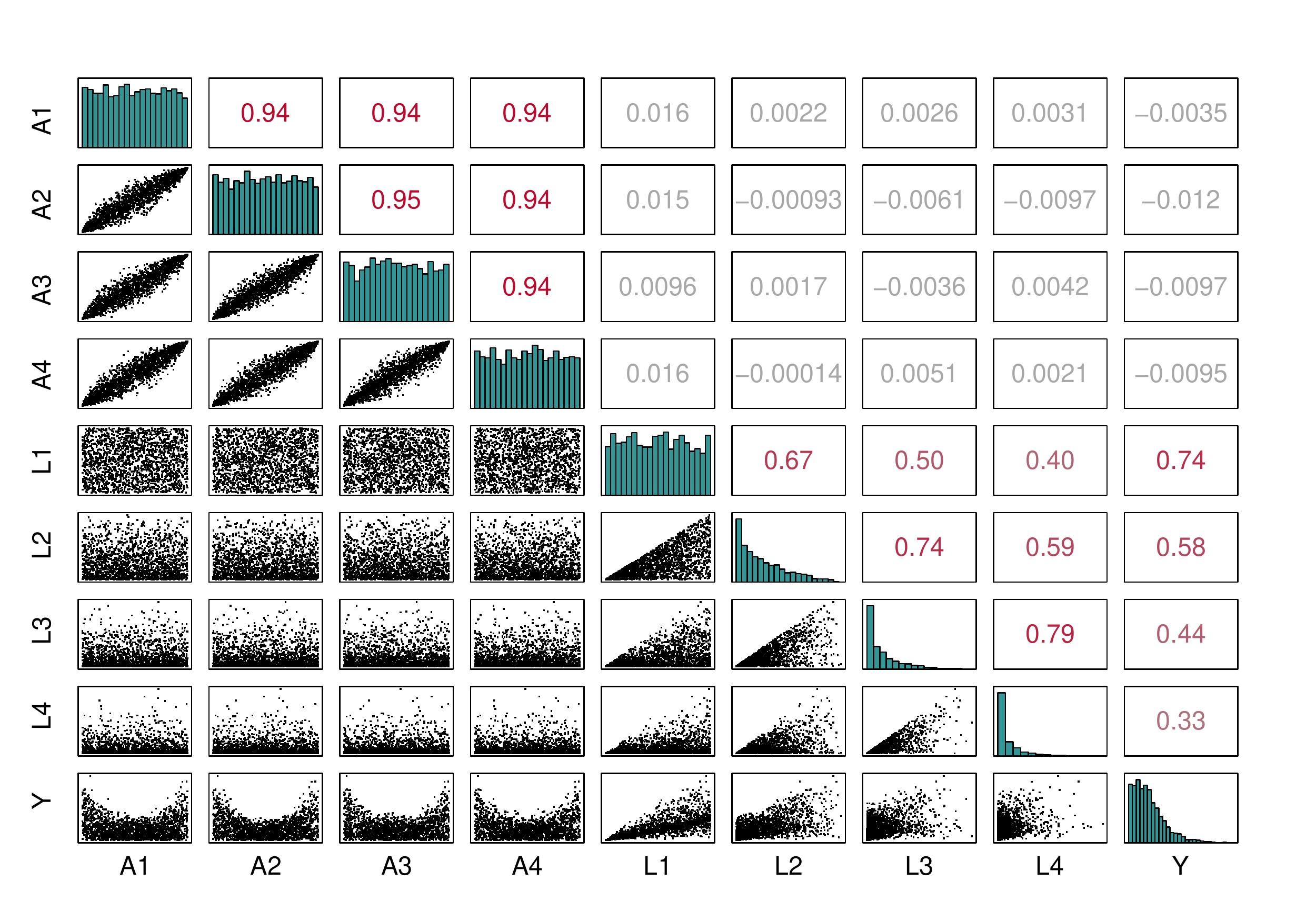}
    \caption{Pairwise plots of the $2000$-size \iid sample for the robot arm model. It gives the histograms of the variables (diagonal), as well as the scatterplots (lower part) and linear correlation coefficients (upper part) between all the variables pairs.}
    \label{fig:scatrobot}
\end{figure}

Since, in this scenario, only an \iid sample is available, the Shapley effects and $\PME$ have been computed using the nearest-neighbor procedure eluded in Section~\ref{sec:estimation} (with an arbitrarily chosen number of nearest neighbors equal to $6$). Figure \ref{fig:ShPMErobot} displays the Shapley effects and PME estimates with a $90\%$-confidence intervals computed on $100$ replications of estimated effects by random selection of $80\%$ of the dataset's observations.
\begin{figure}[!ht]
    \centering
    \includegraphics[width=\textwidth, trim=0cm 2cm 0cm 0cm]{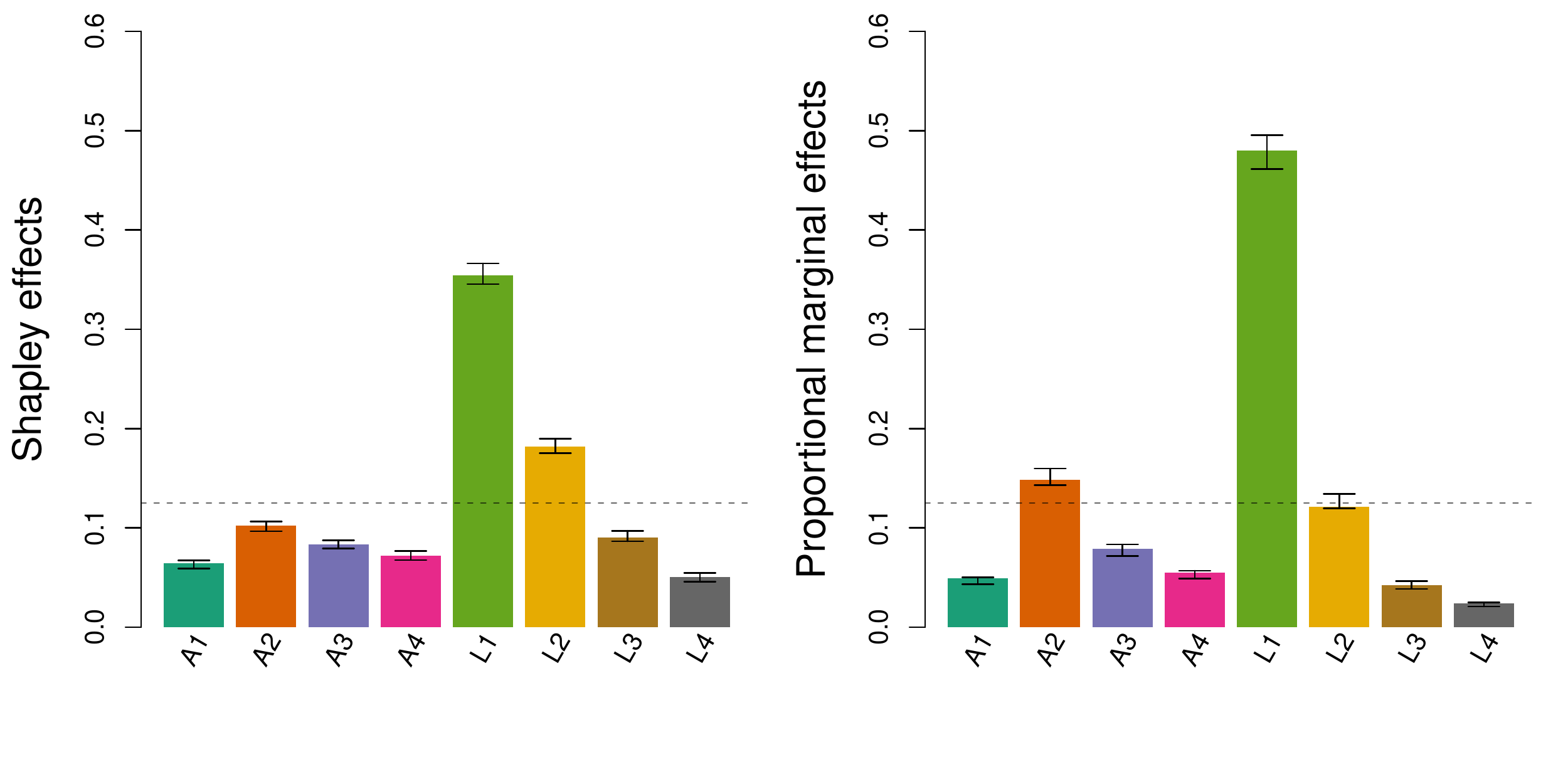}
    \caption{Shapley effects and PME estimators by nearest-neighbor procedure for the robot arm model. The vertical error bars represent the $90\%$-confidence intervals of the estimates. The horizontal grey dashed bar represents the average importance of $1/8$.}
    \label{fig:ShPMErobot}
\end{figure}
According to both effects, the most influential input is $L_1$, with Shapley effect around $35\%$ and $\PME$ around $48\%$ of the output's variance. While both effects seem to agree on the most influential input, they offer a fairly different influence hierarchy, as depicted in \mytabref{tab:robot_influRank}. This different influence hierarchy can be explained by the fairly high correlation between the inputs. For instance, $L_2$ has Shapley effects around $18\%$ while having a linear correlation coefficient with $L_1$ equal to $0.67$, whereas it has a $\PME$ of around $12.1\%$. Additionally, focusing on the angle inputs which are very linearly correlated, one can notice that their Shapley effects are relatively equal, varying between $6\%$ and $10\%$. On the other hand, their $\PME$ grants $A_2$ nearly $15\%$, with reduced influence of the other angles. Only using the Shapley effects did not consider $A_2$ as an above-average influential variable, while the $\PME$ consider it as important. This highlights the more decisive ability of the $\PME$ for influence ranking in situations of highly correlated inputs, where the Shapley effects typically grants a similar output variance share to each correlated input.

\begin{table}[ht!]
    \centering
    \begin{tabular}{ccccc}
    &\multicolumn{2}{c}{Shapley effects}  &\multicolumn{2}{c}{$\PME$}\\
    \midrule
    Influence Rank &Input & Value & Input & Value\\
    \midrule
    1&$L_1$ & $35.4\%$ & $L_1$ & $48\%$ \\
    2&$L_2$ & $18.2\%$ & $A_2$ & $14.9\%$ \\
    3&$A_2$ & $10.2\%$ & $L_2$ & $12.1\%$ \\
    4&$L_3$ & $9\%$ & $A_3$ & $8\%$ \\
    5&$A_3$ & $8.4\%$ & $A_4$ & $5.5 \%$ \\
    6&$A_4$ & $7.2\%$ & $A_1$ & $4.9 \%$ \\
    7&$A_1$ & $6.4\%$ & $L_3$ & $4.2 \%$ \\
    8&$L_4$ & $5.1\%$ & $L_4$ & $2.4 \%$ \\
    \bottomrule
    \end{tabular}
    \caption{Influence hierarchy between the inputs w.r.t. the Shapley effects and the $\PME$, on the robot arm use-case.}
    \label{tab:robot_influRank}
\end{table}
Moreover, one can notice that no input is considered as exogenous, which is reassuring since every input is effectively involved in the computation of the output. However, for both effects, $L_4$ is the least influential, with Shapley effect of around $5\%$ and $\PME$ around $2\%$.

To conclude, this use-case illustrates the more decisive discriminating power of the $\PME$ compared to the Shapley effects, in cases of highly correlated inputs. Overall, the $\PME$ favor the already most influential inputs at the expense of the inputs their are correlated with. This behavior is particularly interested in a screening setting, along with the ability of the $\PME$ to detect exogenous inputs, while maintaining a meaningful interpretation as shares of variance.

\section{Discussion and perspectives}\label{sec:5}
The main contribution of this paper is the application to GSA of the proportional values, an alternate allocation from cooperative game theory. An extension of the original allocation is proposed, in order to be used for Sobol' cooperative games, leading to novel GSA indices: the proportional marginal effects. They fundamentally differ from the Shapley effects in two ways. First, it is proved that they detect exogenous inputs, by granting them zero allocation despite their eventual correlation. Second, they allow for more decisive influence hierarchies than the Shapley effects, especially when inputs are highly correlated. They remain intrinsically interpretable, as shares of variance of the model's output.

It has been shown, by means of analytical toy-cases and use-cases, that the proposed $\PME$ allow for a more complete depiction of the intricacies of black-box models. However, this testing remains quite limited, and does not allow to derive general truth statements on their behavior. More in-depth studies (such as \cite{ioopri19} for the Shapley effects) on challenging use-cases are mandatory in order to comfort their usefulness in practice.

Two estimation schemes are proposed, based on Monte Carlo sampling scheme or given-data using a nearest-neighbor procedure, initially developed for the Shapley effects. Moreover, their computation rely on the same conditional elements' estimations than the Shapley effects: their computation does not require additional sampling or model evaluations. However, the computational burden associated to their estimation remains a drawback. They require the calculation of an exponential number ($2^d-1$) of Sobol indices. The exact same problem has been highlighted for the Shapley effects estimation.

This present work revolves around the idea of finding alternate allocation rules than the Shapley values which, while mainstream, may not be universally suitable. Introducing the framework of random order models, which allows to define allocations by simply choosing a probability mass function over the set of permutations of $D$, allows to further explore the richness of cooperative game theory. The subsequent diversity of properties inherited from the different allocations could be an interesting way to get complementary viewpoints on variance explanation. As seen in this paper, the Shapley effects and the $\PME$ are designed to give different messages, the interest of which depends on the UQ task one is dealing with. While the $\PME$ is a reasonable option for factor fixing and factor prioritization, the Shapley effects provide a tool for model exploration that allows for a good overview of all the inputs that might have an impact on the output, even though it is only due to correlation with other inputs. Other allocations, such as weighted Shapley values \cite{Kalai1987}, may be defined with different specific UQ tasks in mind, allowing for domain-specific tools for more accurate and relevant sensitivity studies.

As discussed in \cite{RAZAVI2021104954} and highlighted in \cite{ioocha22}, machine learning (ML) interpretability and sensitivity analysis bear many resemblances. The $\PME$ are in part inspired from the \emph{proportional marginal value decomposition}, introduced in \cite{Feldman2005} and further studied in \cite{gro07, gro15}, which are importance measures related to linear regression models. Hence, their formal introduction to the field of sensitivity analysis is an expression of the relevant similarities between both fields and the exciting research opportunities that can arise by taking inspiration from either research domains.

Finally, while the cooperative game theory is an evident source for novel ideas, its intricacies remain poorly understood by both the UQ and ML interpretability communities. Cooperative games in general, and the construction of allocations in particular, are inherently player-centric, while UQ and interpretability studies have historically been model-centric. Despite the evident usefulness of cooperative game theory for defining relevant and interpretable tools, further works must be put into justifying their meaningfulness to broaden their use for practical studies.

\appendix
\renewcommand\thefigure{\thesection.\arabic{figure}}



\section{Software and reproducibility of results}\label{app:software}

All the numerical tests have been performed using the \texttt{R} programming language. Every results and figures presented in this paper can be reproduced by means the openly accessible codes in a GitLab repository\footnote{\href{https://gitlab.com/milidris/PME}{https://gitlab.com/milidris/PME}}, as well as details on the packages used.

\section{Proofs}\label{app:proofThm}

\begin{proof}[Proof of Theorem~\ref{thm:PVExtension}]
Let $(D,v)$ be a nonnegative, monotonic cooperative game, let $A \subseteq D$ be a coalition, and denote $a=|A|$ the cardinal of $A$. Denote $\mathcal{S}_A$ the set of permutations of players in $A$. Let $\pi \in \mathcal{S}_A$, and for the sake of clarity, denote $|\pi| = a$, the number of elements in the permutation. 
By monotonicity, one has, $\forall j \in \{ 1, \dots, |\pi| \}$,
$$v\left( C_{j+1}(\pi)\right) \geq v\left( C_{j}(\pi)\right) \geq 0.$$
Moreover, assume that $v\left( C_0(\pi)\right) = v(\emptyset)=0$. Thus, there exists a level $k_{\pi}(v) \in \{0, 1, \dots, |\pi|\}$ such that $v\left( C_{k_{\pi}(v)}(\pi)\right) =0$ and $v\left( C_{k_{\pi}(v)+1}(\pi)\right)>0$. Formally:
\begin{equation*}
k_{\pi}(v) = \begin{cases}
\max \bigl\{ j \in \{ 1, \dots, |\pi| \} : v\left( C_{j}(\pi)\right) = 0 \bigr\} & \text{if } v\left( C_{1}(\pi)\right)=0, \\ 
0 & \text{otherwise.}\end{cases}
\end{equation*}
For the sake of conciseness and readability, the argument $v$ is omitted and the notation $k_{ \pi } := k_{\pi}(v)$ is adopted. Let $(\epsilon_p)_{p\in \mathbb{N}}$ be a sequence such that:
\begin{equation*}
\forall p \in \mathbb{N}, \epsilon_p > 0, \quad \text{and } \quad \underset{p \rightarrow \infty}{\lim} \epsilon_p = 0.
\end{equation*}
Let $\bigl( (D, v_p) \bigr)_{p \in \mathbb{N}}$ be a sequence of positive, monotonic cooperative games defined, for any $p \in \mathbb{N}$ and for any $A \subseteq D$, as:
\begin{equation*}v_p(A) = \begin{cases}
\epsilon_p & \text{if } v(A) = 0, \\
v(A) & \text{otherwise}.\end{cases}
\end{equation*}
Alternatively, one can notice that, $\forall \pi \in \mathcal{S}_A$, $\forall j \in \{1, \dots, |\pi|\}$,
\begin{equation}\label{eq:altervp}
v_p\bigl(C_j(\pi)\bigr) = \begin{cases}
\epsilon_p & \text{if } j \leq k_\pi, \\
v(C_j(\pi)) & \text{otherwise}.\end{cases}
\end{equation}
Let $p \in \mathbb{N}$, and from the recursive definition of the PV (see, Definition~\ref{def:propVal}) of the positive games $(D,v_p)$, one has, for any $i \in D$:
\begin{equation*}
\PV_i=\cfrac{\sum_{\pi \in  \mathcal{S}_{D_{-i}}} \prod_{m=1}^{d-1} v_p\left(C_{m}(\pi)\right)^{-1}}{\sum_{\sigma \in \mathcal{S}_{ D }} \prod_{m=1}^{d} v_p\left(C_{m}(\sigma)\right)^{-1}}.
\end{equation*}
For the sake of conciseness and clarity, for any $\pi \in \mathcal{S}_A$, $\forall A \subseteq D$, let us introduce the following notation:
\begin{equation*}
    \Upsilon^{l}_k(\pi,v)=\begin{cases}
    \prod_{j=k}^{l}  v\left(C_{j}(\pi)\right)^{-1} & \text{if } k \leq l, \\
    1 & \text{otherwise}.
    \end{cases}
\end{equation*}
One then has that, for any $i \in D$:
\begin{equation*}
    \PV_i = \frac{\sum_{\pi \in \mathcal{S}_{D_{-i}}} \Upsilon_1^{d-1} (\pi, v_p)}{\sum_{\sigma \in \mathcal{S}_{D}} \Upsilon_1^{d} (\sigma, v_p)} = \frac{\sum_{\pi \in \mathcal{S}_{D_{-i}}} \Upsilon_1^{k_{\pi}} (\pi, v_p) \Upsilon_{k_{\pi +1}}^{d-1} (\pi, v_p)}{\sum_{\sigma \in \mathcal{S}_{D}} \Upsilon_1^{k_{\sigma}} (\sigma, v_p) \Upsilon_{k_{\sigma+1}}^{d} (\sigma, v_p)}.
\end{equation*}
However, one can notice, from \myeqref{eq:altervp}, that, for any $\pi \in \mathcal{S}_A$, $A \subseteq D$:
\begin{equation*}
    \Upsilon_1^{k_{\pi}}(\pi, v_p) = \prod_{j=1}^{k_{\pi}} v_p\bigl(C_j(\pi)\bigr)^{-1} = \epsilon_p^{-k_{\pi}},
\end{equation*}
leading to:
\begin{equation*}
    \PV_i = \frac{\sum_{\pi \in \mathcal{S}_{D_{-i}}} \epsilon_p^{-k_{\pi}} \Upsilon_{k_{\pi +1}}^{d-1} (\pi, v_p)}{\sum_{\sigma \in \mathcal{S}_{D}} \epsilon_p^{-k_{\sigma}} \Upsilon_{k_{\sigma+1}}^{d} (\sigma, v_p)}.
\end{equation*}
Denote, for any $j \in D$, $k_{\max}^{-j}$ the size of the largest null coalition in $D_{-j}$, i.e.,
$$k_{\max}^{-j} = \underset{A \in \mathcal{P}(D_{-j})}{\text{max }} \left\{|A| : v(A) = 0 \right\},$$
with the convention that $|\emptyset|=0$, and let $k_{\max} = k_{\max}^{-\emptyset}$, the size of the largest null coalition in $D$, and notice that necessarily, 
\begin{equation}\label{eq:ineq_kmaxpi1}
    \forall j \in D , \forall \pi \in \mathcal{S}_{D_{-j}}, k_{\max}^{-j} \leq k_{\max}.
\end{equation}
Moreover, denote, for any $j \in D \cup \{ \emptyset \}$:
$$\mathcal{R}_{\max}^{-j} = \{ \pi \in \mathcal{S}_{D_{-j}} : k_{\pi} = k_{\max}^{-j} \}, \quad \text{and } \quad \mathcal{R}^{-j} = \{ \pi \in \mathcal{S}_{D_{-j}} : k_{\pi} < k_{\max}^{-j} \}$$
and notice that $\mathcal{R}_{\max}^{-j} \cup \mathcal{R}^{-j} = \mathcal{S}_{D_{-j}}$ and moreover that, 
\begin{equation}\label{eq:ineq_kmaxpi2}
    \forall j \in D \cup \emptyset, \forall \pi \in \mathcal{R}^{-j}, \quad k_{\pi} < k_{\max}.
\end{equation}
Again, denote $\mathcal{R}_{\max} = \mathcal{R}_{\max}^{-\emptyset} $ and $\mathcal{R} = \mathcal{R}^{\emptyset}$.

Hence, for any $j \in D$, one has that:
\begin{align*}
    \sum_{\pi \in \mathcal{S}_{D^{-j}}} \epsilon_p^{-k_{\pi}} \Upsilon_{k_{\pi} +1}^{d-1} (\pi, v) &= \sum_{\pi \in \mathcal{R}_{\max}^{-j}} \epsilon_p^{-k_{\pi}} \Upsilon_{k_{\pi} +1}^{d-1} (\pi, v) + \sum_{\pi \in \mathcal{R}^{-j}} \epsilon_p^{-k_{\pi}} \Upsilon_{k_{\pi} +1}^{d-1} (\pi, v) \\
    &= \sum_{\pi \in \mathcal{R}_{\max}^{-j}} \epsilon_p^{-k_{\max}^{-j}} \Upsilon_{k_{\pi} +1}^{d-1} (\pi, v) + \sum_{\pi \in \mathcal{R}^{-j}} \epsilon_p^{-k_{\pi}} \Upsilon_{k_{\pi} +1}^{d-1} (\pi, v)
\end{align*}
and in the particular case of $j = \emptyset$, one has that:
\begin{align*}
     \sum_{\pi \in \mathcal{S}_{D}} \epsilon_p^{-k_{\pi}} \Upsilon_{k_{\pi} +1}^{d} (\pi, v) &= \sum_{\pi \in \mathcal{R}_{\max}} \epsilon_p^{-k_{\max}} \Upsilon_{k_{\pi} +1}^{d} (\pi, v) + \sum_{\pi \in \mathcal{R}} \epsilon_p^{-k_{\pi}} \Upsilon_{k_{\pi} +1}^{d} (\pi, v)\\
     &=\epsilon_p^{-k_{\max}} \left( \sum_{\pi \in \mathcal{R}_{\max}}  \Upsilon_{k_{\pi} +1}^{d} (\pi, v) + \sum_{\pi \in \mathcal{R}} \epsilon_p^{k_{\max}-k_{\pi}} \Upsilon_{k_{\pi} +1}^{d} (\pi, v)\right).
\end{align*}
It entails that:
\begin{equation*}
    \PV_i = \dfrac{\sum_{\pi \in \mathcal{R}_{\max}^{-i}} \epsilon_p^{k_{\max}-k_{\max}^{-i}} \Upsilon_{k_{\pi} +1}^{d-1} (\pi, v) + \sum_{\pi \in \mathcal{R}^{-i}} \epsilon_p^{k_{\max}-k_{\pi}} \Upsilon_{k_{\pi} +1}^{d-1} (\pi, v)}{\sum_{\sigma \in \mathcal{R}_{\max}}  \Upsilon_{k_{\sigma} +1}^{d} (\sigma, v) + \sum_{\sigma \in \mathcal{R}} \epsilon_p^{k_{\max}-k_{\sigma}} \Upsilon_{k_{\sigma} +1}^{d} (\sigma, v)}.
\end{equation*}
From \myeqref{eq:ineq_kmaxpi2}, one can notice, for any $j \in D \cup \{ \emptyset \}$:
$$\underset{p \rightarrow \infty}{\text{lim}} \sum_{\pi \in \mathcal{R}^{-j}} \epsilon_p^{k_{\max}-k_{\pi}} \Upsilon_{k_{\pi} +1}^{d-|j|} (\pi, v) = 0$$
and additionally, from \myeqref{eq:ineq_kmaxpi1}, notice that for any $j \in D$:
$$\underset{p \rightarrow \infty}{\text{lim}} \sum_{\pi \in \mathcal{R}_{\max}^{-j}} \epsilon_p^{k_{\max}-k_{\max}^{-j}} \Upsilon_{k_{\pi} +1}^{d-1} (\pi, v) =\begin{cases}
\sum_{\pi \in \mathcal{R}_{\max}^{-j}} \Upsilon_{k_{\pi} +1}^{d-1} (\pi, v) & \text{if } k_{\max} = k_{\max}^{-j} \\
0 & \text{otherwise.}
\end{cases}$$
Denote:
\begin{equation*}
    \PVext \bigl((D,v) \bigr) = \underset{p \rightarrow \infty}{\text{lim }} \PV\bigl(D, v_p \bigr) ,
\end{equation*}
and notice that, for any $i \in D$:
\begin{equation*}
    \PVext_i = \begin{cases}
    \dfrac{\sum_{\pi \in \mathcal{R}_{\max}^{-i}} \Upsilon_{k_{\pi} +1}^{d-1} (\pi, v)}{\sum_{\sigma \in \mathcal{R}_{\max}} \Upsilon_{k_{\sigma} +1}^{d-1} (\sigma, v)} & \text{if } k_{\max} = k_{\max}^{-i}, \\
    0 & \text{otherwise}.
    \end{cases}
\end{equation*}
For any $j \in D$, the condition $k_{\max} = k_{\max}^{-j}$ is equivalent to the existence of a coalition $A \subseteq D_{-j}$ such that $|A| = k_{\max}$ and $v(A) = 0$. On the other hand, the complement of this condition is that $j$ must be in every coalition $A \subseteq D$ such that $|A| = k_{\max}$ and $v(A) = 0$, leading to the condition for which $\PVext_j=0$.

For any $j \in D \cup \{ \emptyset\}$, and assuming that $k_{\max}^{-j} = k_{\max}$, one can notice that $\mathcal{R}_{\max}^{-j}$ only contains the permutations $\pi \in \mathcal{S}_{D_{-j}}$ such that $v\left(C_{k_{\max}}(\pi)\right) = 0$, and by monotonicity, this implies that for any $\pi \in \mathcal{R}_{\max}^{-j}$:
$$v\left(C_1(\pi)\right) = v\left(C_2(\pi)\right) = \dots = v\left(C_{k_{\max}-1}(\pi)\right) = v\left(C_{k_{max}}(\pi)\right)=0,$$
and that for $k_{\max} < k \leq |\pi|$,
$$v\left(C_k(\pi)\right) >0.$$
For any $j \in D \cup \{\emptyset\}$, denote $\mathcal{K}_{-j} = \left\{A \subseteq D_{-j} : v(A) = 0~\text{and}~ |A| = k_{\max} \right\}$, and notice that $\mathcal{R}_{\max}^{-j}$ is necessarily composed of permutations having permutations of elements in $\mathcal{K}_{-j}$ as their first $k_{\max}$ elements. In other words, for every $\pi \in \mathcal{R}_{\max}^{-j}$,
$$C_{k_{\max}}(\pi) \in \mathcal{K}_{-j}.$$
Thus, for any $j\in D \cup \emptyset$, one has that:
\begin{align*}
    \sum_{\pi \in \mathcal{R}_{\max}^{-j}} \Upsilon_{k_{\pi} +1}^{d-1} (\pi, v) &= \sum_{A \in \mathcal{K}_{-j}}  k_{\max}! \sum_{\pi \in \mathcal{S}_{D_{-j} \setminus A}} \Upsilon_1^{|\pi|}(\pi, v_A) \\
    &=k_{\max}! \sum_{A \in \mathcal{K}_{-j}} \sum_{\pi \in \mathcal{S}_{D_{-j} \setminus A}} \prod_{k=1}^{|\pi|} v(A \cup C_k(\pi))^{-1} \\
    &= k_{\max}! \sum_{A \in \mathcal{K}_{-j}} R(D_{-j}\setminus A, v_A)^{-1}
\end{align*}
where for any $B \subseteq D \setminus A$, $v_A(B) = v(A \cup B)$, and using results from \cite{Feldman2007} on the ratio potential. This leads to the following rewriting of $\PVext$, for any $i \in D$:
\begin{equation*}
    \PVext_i = \begin{cases}
    0 & \text{if } \forall A \in \mathcal{K}, i \in A, \\
    \dfrac{ \sum_{A \in \mathcal{K}_{-i}} R(D_{-i}\setminus A, v_A)^{-1}}{ \sum_{A \in \mathcal{K}} R(D\setminus A, v_A)^{-1}} & \text{otherwise.}
    \end{cases}
\end{equation*}
Finally, notice that for any positive game $(D,v)$, \ie, where $v$ is positively valued, then necessarily, for any permutation and sub-permutations $\pi$ of players $k_{\pi} = k_{\max}=0$. Moreover, for any $j \in D \cup \{ \emptyset\}$, $\mathcal{K}_{-j} = \emptyset$, and for any $i \in D$, 
$$\PVext_i = \frac{R(D, v)}{R(D_{-i}, v)} = \PV_i,$$
and hence the allocation $\PVext\bigl( (D,v) \bigr)$ is a continuous extension of $\PV\bigl((D,v) \bigr)$ to cooperative games with nonnegative value function.
\end{proof}

\begin{proof}[Proof of Lemma~\ref{lme:StExo}]
Let $A \subseteq D$. First, focus on the implication
$$\St_A= 0 \implies  G(X) = \E{G(X) \mid X_{\overline{A}}} \text{ a.s.} $$
If $\St_A=0$, then necessarily,
$$\V{G(X) \mid X_{\overline{A}}}:=\E{\left(G(X) - \E{G(X) \mid X_{\overline{A}}}\right)^2 \mid X_{\overline{A}}}=0 \text{ a.s.}$$
which can only be attained, by non-negativity of the squared distance, if 
$$G(X) =  \E{G(X) \mid X_{\overline{A}}} \text{ a.s.}$$
which proves the implication.

Now assume that $G(X) =  \E{G(X) \mid X_{\overline{A}}} \text{ a.s.}$. Then necessarily,
$$\V{G(X) \mid X_{\overline{A}}} = 0 \text{ a.s.} $$
and thus $\St_A=0$, which proves the converse implication, and the equivalence stated in Lemma~\ref{lme:StExo}.
\end{proof}

\begin{proof}[Proof of Proposition~\ref{prop:EquivExoPME}]


Assume that $X_E$ is the largest subset of $X$ of $L^2$-exogenous inputs to $G$. This entails, by Definition~\ref{def:exogInput}, $\forall B \subseteq E, \exists f \in L^2(P_{X_{\overline{B}}})$ such that:
$$ G(X) = f(X_{\overline{B}}) \text{ a.s.}$$
Recall that for any $B\subset E$, the conditional expectation of $G(X)$ w.r.t. $X_{\overline{B}}$ is the unique projection defined as:
\begin{equation*}
    \E{G(X) \mid X_{\overline{B}}} = \underset{h \in L^2\left(P_{X_{\overline{B}}}\right)}{\text{argmin }} \E{\left(G(X) - h\left(X_{\overline{B}}\right)\right)^2},
\end{equation*}
One can notice that, since $f \in L^2(P_{X_{\overline{B}}})$ and $Y = f(X_{\overline{B}}) \text{ a.s.}$, then it necessarily minimizes the projection of $G(X)$ onto $L^2(P_{\overline{B}})$, leading to
$$ G(X) = f(X_{\overline{B}}) = \E{G(X) \mid X_{\overline{B}}} \text{ a.s.} $$
and by Lemma~\ref{lme:StExo}, it entails that, for every $B \subseteq E$, $\St_B=0$.
Furthermore, $X_E$ being the largest subset of $X$ being $L^2$-exogenous, it entails that for any $A \subseteq D$ such that $|A| \geq |E|$ and $A \neq E$, $\nexists f \in L^2\left(P_{X_{\overline{A}}}\right)$ s.t. $G(X) = f(X_{\overline{A}})$. Then, necessarily for any coalition $A \neq E$ of size larger or equal than $E$,
\begin{equation*}
    \E{G(X) \mid X_{\overline{A}}} \not = G(X) \text{ a.s.}  \iff \St_{A} > 0.
\end{equation*}
Hence, $E$ is the largest set of inputs with $\St_E=0$. Then, necessarily, $\mathcal{K} = \{ E \}$ where $\mathcal{K}$ is defined as in Theorem~\ref{thm:PVExtension}. Then $\forall i  \in E, \forall A \in \mathcal{K}, i \in A$. Hence, it entails that $\PME_i=0$ for any $i \in E$, and, furthermore, $\forall i \in \bar{E}$, $i$ is not contained in $E \in \mathcal{K}$. Finally, by Theorem~\ref{thm:PVExtension}, $\PME_i>0$ for any $i \in E$.
\end{proof}

We are grateful to Nicolas Bousquet (EDF R\&D), Fabrice Gamboa (IMT) and Christophe Labreuche (Thales) for their helpful comments.

\bibliographystyle{plain}
\bibliography{bib/references}

\end{document}